\documentclass[a4paper,10pt]{book}
\usepackage[english]{babel}
\usepackage{amsmath}
\usepackage{amsthm}
\usepackage{amssymb,units,mathtools}
\usepackage[T1]{fontenc}
\usepackage{commath}
\usepackage{vwcol}
\usepackage[a4paper]{geometry}
\usepackage{xfrac}

\usepackage[utf8]{inputenc}
\usepackage{graphicx}
\usepackage{graphics}
\usepackage{mathrsfs}
\usepackage{xfrac}
\usepackage{multicol}
\usepackage{yhmath}
\usepackage{multirow}
\usepackage{graphicx}
\usepackage{mathrsfs}
\usepackage{xfrac}
\usepackage{multicol}
\usepackage{multirow}
\usepackage{float}
\usepackage{relsize}
\usepackage[overload]{empheq}
\usepackage{epstopdf}
\usepackage{epsfig}
\usepackage[T1]{fontenc}
\usepackage[nobysame]{amsrefs}
\usepackage{xfrac}
\usepackage{float}
\usepackage{enumitem}
\usepackage[symbol]{footmisc}
\setenumerate[1]{label={\roman*)}}

\newcommand{\N}{\mathbb{N}}
\newcommand{\R}{\ensuremath{\mathbb{R}}}

\usepackage{pst-node}
\usepackage{tikz-cd} 
\usepackage{xcolor}

\theoremstyle{plain}
\newtheorem{Theorem}{Theorem}[section]
\newtheorem{Proposition}[Theorem]{Proposition}
\newtheorem*{Theorem*}{Theorem}
\newtheorem{Corollary}[Theorem]{Corollary}
\newtheorem{Lemma}[Theorem]{Lemma}

\theoremstyle{Remark}
\newenvironment{abstract}{}{}
\usepackage{abstract}

\newtheorem{Example}[Theorem]{Example}

\theoremstyle{Definition}
\newtheorem{Definition}[Theorem]{Definition}
\usepackage{pst-node}
\usepackage{fancyhdr}
\usepackage{xcolor}
\renewcommand{\S}{\ensuremath{\mathbb{S}}}
\renewcommand{\H}{\ensuremath{\mathbb{H}}}
\newcommand{\V}{\ensuremath{\mathbb{V}}}
\newcommand{\C}{\ensuremath{\mathbb{C}}}
\newcommand{\Z}{\ensuremath{\mathbb{Z}}}

\usepackage{color}
\usepackage{titlesec}
\usepackage{chapterbib}
\titleformat{\chapter}[display]
{\normalfont\huge\bfseries}{\chaptertitlename\ \thechapter}{20pt}{\Huge}
\titlespacing*{\chapter}{0pt}{-10pt}{20pt}
\usepackage{afterpage}

\usepackage[colorlinks = true]{hyperref}

\hypersetup{
	linkcolor=black,
	citecolor = black,
	urlcolor=black,
	filecolor=black
}

\newcommand{\lspan}{\ensuremath{\pmb{\langle}}}
\newcommand{\rspan}{\ensuremath{\pmb{\rangle}}}

\title{\textbf{TFG Matemàtiques\\ On the Canonical Contact Structure of \\the Space of Null Geodesics of a Spacetime}}
\author{
	Adrià Marín Salvador \\
}
\date{}
\begin{document}
	\begin{titlepage} % Suppresses ers and footers on the title page
		
		\centering % Centre everything on the title page
		
		%\scshape  Use small caps for all text on the title page
		
		\vspace*{\baselineskip} % White space at the top of the page
		
		%------------------------------------------------
		%	Title
		%------------------------------------------------
		\includegraphics[width = 0.5\textwidth]{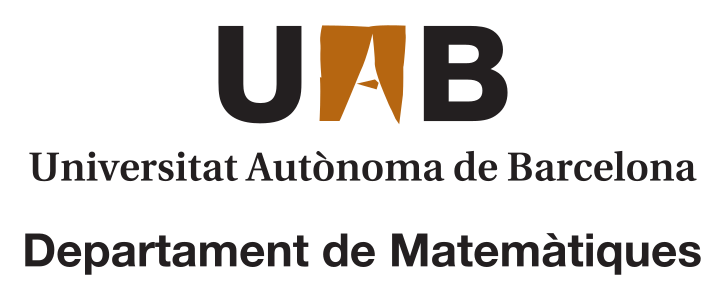}
		
		\rule{0.7\textwidth}{0.8pt} % Thin horizontal rule
		
		\vspace{0.65\baselineskip} % Whitespace above the title
		
		{\LARGE \textbf{Bachelor's Thesis in Mathematics}} % Title
		
		\vspace{0\baselineskip} % Whitespace below the title
		
		\rule{0.7\textwidth}{0.8pt}\vspace*{-\baselineskip}\vspace{3.2pt} % Thin horizontal rule
		
		\vspace{1.8\baselineskip} % Whitespace after the title block
		
		%------------------------------------------------
		%	Subtitle
		%------------------------------------------------
		
		\LARGE \textbf{ON THE CANONICAL CONTACT STRUCTURE \\ OF THE SPACE OF NULL GEODESICS \\OF A SPACETIME} % Subtitle or further description
		
		\vspace*{-.5\baselineskip} % Whitespace under the subtitle
		
		%------------------------------------------------
		%	Editor(s)
		%------------------------------------------------

		\vspace{0.8\baselineskip} % Whitespace before the editors
		
		\includegraphics[width = \textwidth, trim=4 200 4 4,clip = true]{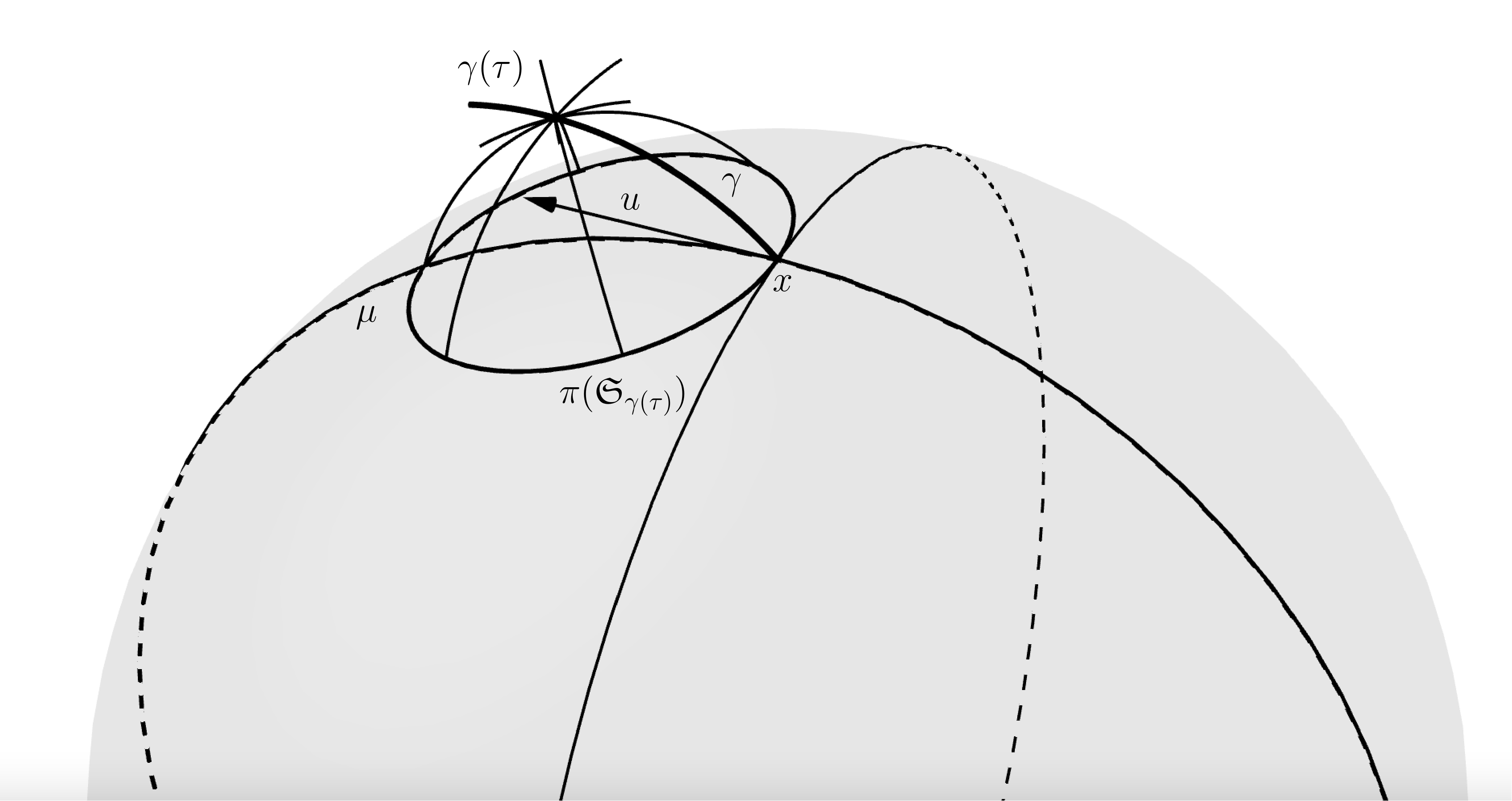}
		
		{\scshape\Large Adrià Marín Salvador \\} % Editor list
		
		\vspace{.7\baselineskip} % Whitespace below the editor list
		
		{\scshape\large Supervised by Dr. Roberto Rubio Núñez\\ Cosupervised by Dr. Francisco Presas Mata}% Editor affiliation
		
		\vfill % Whitespace between editor names and publisher logo
		
		%------------------------------------------------
		%	Publisher
		%------------------------------------------------

		\vspace{0.3\baselineskip} % Whitespace under the publisher logo
		
		% Publication year
		
		{\large July 2021}
	\end{titlepage}
		\pagenumbering{gobble}
\begin{abstract}
	The space of null geodesics of a spacetime carries a canonical contact structure which has proved to be key in the discussion of causality in spacetimes. However, not much progress has been made on its nature and not many explicit calculations for specific spacetimes can be found over the literature. We compute the spaces of null geodesics and their canonical contact structures for the manifold $\S^2\times\S^1$ equipped with the family of metrics $\{g_c = g_\circ-\frac{1}{c^2}dt^2 \}_{c\in\N^+}$. We obtain that these are the lens spaces $L(2c,1)$ and that the contact structures are the pushforward of the canonical contact structure on $ST\S^2\cong L(2,1)$ under the projection map. We also study the applicability of Engel geometry on the discussion of three-dimensional spacetimes. We show that, for a particular type of three-dimensional spacetimes, one can obtain the space of null geodesics and its contact structure solely from the information of the Lorentz prolongation of the spacetime. We present an approach that makes use of this result to recover the spacetime from its space of null geodesics and skies.
\end{abstract}

\chapter*{\huge Acknowledgements}
I am deeply grateful to my supervisors Dr. Roberto Rubio and Dr. Francisco Presas. This work started as a collaboration with Fran at the ICMAT supported by a Severo Ochoa - Introduction to Research grant. I am thankful to ICMAT for providing me with such opportunity. I cannot put into words how much Fran (and also his students) helped me and supported me during my time at ICMAT, and also during the production of this thesis. Fran's ideas and comments have been essential throughout the whole process.

After coming back from ICMAT, Roberto accepted to continue the work on this topic and to supervise my thesis. I am indebted to him for agreeing on embarking on this project, which fell outside of his area of expertise. Roberto's willingness to help me, guide me and support me has been crucial during my work. His numerous contributions have also made this dissertation undoubtedly better. Roberto also accepted to support me when I received a \textit{Beca de Colaboración}, which has also funded me throughout the production of this thesis.

Roberto's and Fran's different (but extremely valuable) ways of viewing and understanding mathematics have enriched both this dissertation and my mathematical skills. I cannot stress enough how key their support has been in this early stage of my mathematical career.

I also appreciate all the contributions made by friends and colleagues that have substantially improved this work. I would like to specifically mention my colleagues Carles Falcó, Jaime Pedregal, Miquel Saucedo and Sergio Serrano de Haro. I would also like to acknowledge Teo Gil Moreno de Mora, with whom I shared my experience at ICMAT and with whom I started developing the ideas that led to this dissertation.

	\tableofcontents
	\chapter*{Introduction}
	\addcontentsline{toc}{chapter}{Introduction}
	\setcounter{page}{1}
	\pagenumbering{arabic}
	
At the turn of the twentieth century, inspired by the laws of nature that A. Einstein later framed under the name of \textit{general relativity} \cite{einstein}, mathematicians developed the tools of pseudo-Riemannian and Lorentzian geometry and introduced the concept of Minkowski spacetime \cite{poincare, minkowski}. This is the pseudo-Riemannian manifold $\R^4$ with metric $\eta = dx_1^2+dx_2^2+dx_3^2-dx_4^2$. The negative eigendirection of the metric depicts time, while the span of the others characterizes the spatial components of the spacetime. 
	
	The theory rapidly evolved past physical meaning and became purely mathematical, and the concept of Minkoswki spacetime was generalised. A spacetime is a Lorentzian manifold (that is, a pseudo-Riemannian $n$-manifold with signature $(n-1,1)$) in which one can choose a vector field with negative length at each point, called a choice of future. At each point of the manifold, the metric defines two hemicones of vectors of length zero, which are called the null vectors of the spacetime. 
	
	In the 1980s, influenced by the work of R. Penrose \cite{penrose1, penrose2}, R. Low introduced the space of unparametrized geodesics with null tangent vectors at all points, called the space of null geodesics of a spacetime, and studied its topology and geometry \cite{low1,low2,low3,low4, low5}. When this space is a differentiable manifold, Low discovered the existence of a canonical contact structure \cite{low5}. A contact structure is a distribution of hyperplanes (that is, a smooth choice of a hyperplane on every tangent space) that is maximally non-integrable. The contact structure on the space of null geodesics satisfies that every sky (the set of geodesics going through a particular point of the spacetime) is everywhere tangent to the distribution.
	
	The study of the contact structure on the space of null geodesics has proved to be essential in the theory, yielding important results on causality, for instance, providing obstructions to two points in the spacetime being related by a curve with non-positive tangent vectors  \cite{chernov, nonneg, natario}. However, apart from \cite{bautista}, there has not been much progress in the understanding of the nature of this structure and on the possibility of recovering the spacetime solely from its space of null geodesics. In addition, not many explicit calculations of spaces of null geodesics and their contact structures can be found in the literature. The present work aims to contribute to these directions. 
	
	In the first part, we compute the spaces of null geodesics and corresponding contact structures of the manifold $\S^2\times\S^1$ for the family of Lorentzian metrics $\{g_c = g_\circ-\frac{1}{c^2}dt^2 \}_{c\in\N^+}$, where $g_\circ$ is the round metric on $\S^2$ and $t$ is the coordinate on $\S^1$. The spacetimes $(\S^2\times\S^1, g_c)$ provide an interesting example because their spaces of null geodesics are not, in general, equivalent to the canonical contact structure of a unit tangent bundle. The latter is always the case whenever there is a global hypersurface $C$ such that every curve with tangent vectors of non-positive length intersects $C$ exactly once. Such a surface does not exist for any of the spacetimes $(\S^2\times\S^1, g_c)$, since that would imply that $\S^2\times\S^1$ is non-compact \cite{BerS}.
	
	\newpage
	More precisely, by developing a quaternionic approach to the Hopf fibration, we show the following result, see Theorem \ref{Bigone}. 
	
	\begin{Theorem*}
 Let $\mathcal{N}_c$ be the space of null geodesics on $\S^2\times\S^1$ under the metric $g_c$. Then, $\mathcal{N}_c$ is diffeomorphic to the lens space $L(2c,1)$, that is,
		\[
		\mathcal{N}_c\cong L(2c,1).
		\]
	\end{Theorem*}
	
	The lens space $L(p,1)$ is the manifold obtained by quotienting the three-sphere $\S^3\subset \C\times\C$ by the finite $\Z_p$-action generated by $(z_1,z_2)\mapsto(e^{2\pi i /p}z_1, e^{2\pi i/p}z_2)$. In addition, the canonical contact structure on $\mathcal{N}_1\cong L(2,1)\cong ST\S^2$ is shown to be the canonical contact structure $\chi$ on $ST\S^2$. We also show how the spaces $L(2c,1)$ can be obtained by quotienting the unit tangent bundle $ST\S^2$ by a finite $\Z_c$-action, which recovers $L(4,1)\cong ST\R P^2$, see \cite{unitprojective}. We totally characterize the contact structure on $\mathcal{N}_c\cong L(2c,1)$ as follows, see Theorem \ref{contactstructure}.
	
	\begin{Theorem*}
		Let $r:ST\S^2\to L(2c,1)\cong\mathcal{N}_c$ be the canonical projection. Let $\chi$ be the canonical contact structure on $ST\S^2$. Then, the contact structure on $\mathcal{N}_c$ is
		\[
		\mathcal{H}_c = r_*\chi.
		\]
	\end{Theorem*}
	
	The second part of this dissertation focuses on the recovery of a spacetime given its contact manifold of null geodesics. We discuss the three-dimensional case, for which we make use of Engel geometry. Given a four-dimensional manifold $M$, an Engel structure on $M$ is a rank-two distribution $\mathcal{D}$ satisfying that $\mathcal{E}: = [\mathcal{D}, \mathcal{D}]$ is a rank-three distribution such that $[\mathcal{E}, \mathcal{E}] = TM$, see \cite{engeldeform}. It can be shown that an Engel structure defines a line field $\mathcal{W}$, which completes the flag $\mathcal{W}\subset\mathcal{D}\subset\mathcal{E}\subset TM$. In the 1920s, E. Cartan discovered how, given a three-dimensional manifold equipped with a contact structure, one can define a canonical Engel manifold, nowadays called its Cartan prolongation. Following these ideas, R. Casals, J. Pérez, A. del Pino and F. Presas  defined a canonical Engel structure coming from a Lorentzian three-manifold, called its Lorentz prolongation \cite{existenceh}. We explore how Engel structures, prolongations, and their inverse operations (deprolongations) can be helpful in recovering a Lorentzian manifold from its space of null geodesics and its canonical contact structure.   
	
We study the case of separable spacetimes, which we define as those for which, locally, the spatial components of the metric are invariant under the negative eigendirection within the spacetime, and conversely, see Definition \ref{separable}. Our main contribution is the following result, which can be found in Theorems \ref{RecoverManifold} and \ref{Bigthree}.

\begin{Theorem*}
	Let $M$ be a three-dimensional separable spacetime. Let $\mathcal{P}C$ be the Lorentz prolongation of $M$, with Engel flag $\mathcal{W}\subset\mathcal{D}\subset\mathcal{E}\subset TM$. Then
	\[
\mathcal{N} = \mathcal{P}C/\mathcal{W}.	
	\]
	In addition, if $\mathcal{N}$ is a differentiable manifold and the projection $p:\mathcal{P}C\to \mathcal{P}C/\mathcal{W}$ is a submersion, the contact structure $\mathcal{H}$ on $\mathcal{N}$ is given by
	\[
	\mathcal{H} = p_*\mathcal{E}.
	\]
\end{Theorem*}
	 
	This theorem allows us to obtain the contact manifold of null geodesics solely with the information given by the Engel flag on the four-dimensional manifold. In addition, we discuss how this approach can be useful in recovering a Lorentzian manifold from its space of null geodesics and its skies. Further research is needed to find suitable hypotheses on the contact manifolds ensuring that the arguments hold, but our initial results are encouraging.
	
\vspace{1cm}	
	This thesis is structured as follows. In Chapter \ref{Prels}, we introduce the main concepts and results in contact and Lorentzian geometry, the space of null geodesics and Engel structures. Chapter \ref{Model} studies the model $\S^2\times \S^1$ with the family of Lorentzian metrics $\{g_c = g_\circ-\frac{1}{c^2}dt^2 \}_{c\in\N^+}$, making use of the division algebra of quaternions and their relation to $ST\S^2$ and the Hopf fibration. Finally, in Chapter \ref{ProlandDeprolChap}, we explore the applicability of Engel geometry in the theory, arguing how one can obtain the spaces of null geodesics of a separable spacetime and its contact structure from the Lorentz prolongation of the spacetime. We also explore how this procedure allows us to obtain a Lorentzian manifold with a particular space of null geodesics and contact structure.
\chapter{Preliminaries}
\label{Prels}
\section{Introduction to Contact Geometry}
\label{IntroToContact}

In Section \ref{IntroToContact} we recall the main definitions and results of contact geometry and topology that will be needed throughout this work. This section does not intend to be a thorough description of the field, we refer to \cite{geiges} for further details.

Let $M$ be a differentiable manifold of dimension $m$ and let $TM$ denote its tangent bundle.

\begin{Definition}
	\label{Codimone}
	A \textbf{codimension one distribution} (or \textbf{field of hyperplanes}) on $M$ is a smooth subbundle $\xi\subset TM$ of codimension one. We will write $X\in\xi$ to denote that $X$ is a smooth section of $TM$ with $X(x)\in\xi_x$ for all $x\in M$. 
	A rank $n<m$ \textbf{distribution} on $M$ is a smooth subbundle $\xi\subset TM$ of rank $n$.
\end{Definition}

It should be clear what is meant by \textit{smooth} in Definition \ref{Codimone}. We demand that, for all $x\in M$, there exists a neighbourhood $U\subseteq M$ of $x$ and $n$ vector fields $X_1,\ldots,X_{n}\in\mathfrak{X}(U)$ that span $\xi|_{U}$, that is, such that $\xi_y = \lspan X_1(y),\ldots,X_{n}(y)\rspan $ for all $y\in U$. All through this work, all objects on a manifold will be assumed to be smooth, unless stated otherwise.

A key observation that will be useful in defining contact structures is that one can always regard codimension one distributions as the kernel of a one-form on $M$, at least locally.

\begin{Lemma}
	\label{kerneloneform}
	\cite[Lem. 1.1.1]{geiges} Locally, a codimension one distribution $\xi$ can be written as the kernel of a differential one-form $\alpha$. In addition, it is possible to write $\xi = \ker\alpha$ globally if and only if $\xi$ is coorientable,  that is, the quotient bundle $TM/\xi$ is trivial.
\end{Lemma}

For the rest of this section, all fields of hyperplanes will be assumed to be coorientable unless specified. We can now define what a contact structure on an odd-dimensional manifold $M$ is.

\begin{Definition}
	Let $M$ be a $(2n+1)$-dimensional manifold. Let $\xi = \ker\alpha\subset TM$ be a codimension one distribution on $M$ such that 
\[
\alpha\wedge(d\alpha)^n\neq 0,
\]	
that is, the top form $\alpha\wedge(d\alpha)^n$ vanishes nowhere. The one-form $\alpha$ is called a \textbf{contact form}, and $\xi = \ker\alpha$ is called a \textbf{contact distribution}. The pair $(M,\xi)$ is a \textbf{contact manifold}.
\end{Definition}

Note that if $\xi$ is a contact distribution on $M$ such that $\xi = \ker\alpha$ globally, the top form $\alpha\wedge(d\alpha)^n$ is a volume form on $M$, meaning that $M$ must be orientable.

Let us now provide some intuition for the definition of contact distributions. Let $\xi = \ker\alpha$ be a field of hyperplanes on $M$. Then, $\xi$ is said to be integrable if, through any point $x\in M$, there exists a submanifold $N\subset M$ with the property that $T_yN = \xi_y$ for all $y\in N$. Frobenius' Theorem gives a characterisation of such distributions.
\begin{Theorem}[Frobenius' Theorem]
	\label{Frobenius}
	\cite[Prop 1.59 and Thm. 1.60]{warner} The following conditions are equivalent.
	\begin{enumerate}
		\item The distribution $\xi$ is integrable.
		\item For any $X,Y\in\xi$, it holds that $[X,Y]\in\xi$, where $[-,-]$ denotes the Lie bracket.
	\end{enumerate}
\end{Theorem}
Condition $ii)$ of Frobenius' Theorem can be rewritten in terms of the defining one-form $\alpha$ of a field of hyperplanes, which gives the following result.
\begin{Corollary}
	The codimension one distribution $\xi = \ker\alpha$ is integrable if and only if 
	\[
	\alpha\wedge d\alpha \equiv 0,
	\]
	where $\equiv$ denotes that the differential form vanishes everywhere.
\end{Corollary}
\begin{proof}
	Assume $\alpha\wedge d\alpha\equiv 0$ holds and let $X,Y\in\xi$, that is, $\alpha(X) =\alpha(Y) = 0$. Then,
	\[
	\alpha([X,Y]) = X\alpha(Y)-Y\alpha(X)-d\alpha(X,Y) = -d\alpha(X,Y).
	\]
	Note also that
	\[
	0 = \iota_X(\alpha\wedge d\alpha) = \alpha\wedge\iota_X d\alpha,
	\]
	and hence
	\[
	0 = \iota_Y\big(\iota_X(\alpha\wedge d\alpha)\big) = \iota_Y(\alpha\wedge\iota_X d\alpha) = -\alpha\wedge d\alpha(X,Y),
	\]
	which implies that $d\alpha(X,Y) = 0$. Hence, $[X,Y]\in\xi$, as needed.
	
	Assume now that condition $ii)$ in Frobenius' Theorem holds. Let $x\in M$ and $U$ be an open subset of $M$ containing $x$ such that one can define vector fields $X_1,\ldots,X_{m-1}\in\xi$, and $Y\in\mathfrak{X}(U)$ satisfying $T_xM = \lspan X_1(y)\ldots,X_{m-1}(y),Y(y) 
	\rspan$ for all $y\in U$. It is clear that, for any $i,j,k\in\{1,\ldots,m-1\}$, one has $(\alpha\wedge d\alpha)(X_i,X_j,X_k) = 0$, since $\alpha(X_i)$ vanishes for all $i = 1,\ldots, m-1$. Now,
	\[
	(\alpha\wedge d\alpha)(X_i,X_j,Y) = \frac{\alpha(Y)}{2}\Big(d\alpha(X_i,X_j)-d\alpha(X_j,X_i)\Big) = \frac{\alpha(Y)}{2}\Big(-\alpha([X_i,X_j])+\alpha([X_j,X_i])\Big) = 0.
	\]
	Hence, $(\alpha\wedge d\alpha)(y) = 0$ for all $y\in U$ and, since $x$ is arbitrary, we find that $\alpha\wedge d\alpha = 0$.
\end{proof}

Contact distributions are, in a sense, the opposite of integrable distributions, for which they are sometimes referred to as \textbf{maximally non-integrable distributions}, and the condition $\alpha\wedge(d\alpha)^n \neq 0$ is called the \textbf{maximally non-integrable condition}.
\begin{Example}
	Let $M = \R^{2n+1}$ with cartesian coordinates $(x_1,y_1,\ldots,x_n,y_n,z)$ and define the one-form
	\[
	\alpha = dz + \sum\limits_{j= 1}^n x_j dy_j.
	\]
	Let us compute $d\alpha = \sum\limits_{j= 1}^ndx_j\wedge dy_j$, and hence
	\[
	\alpha\wedge(d\alpha)^n = dz\wedge dx_1\wedge dy_1\wedge\cdots\wedge dx_n\wedge dy_n \neq0,
	\]
	which implies that $\alpha$ is a contact form. The field of hyperplanes $\xi = \ker\alpha$ is known as the \textbf{standard contact distribution} on $\R^{2n+1}$.
\end{Example}
\begin{Example}
	\label{ExampleUTB}
	Let $(S, g)$ be a Riemannian surface and consider its unit tangent bundle
	\[
	STS = \{u\in TS\ |\ g(u,u) = 1\},
	\]
	which inherits a structure of $\S^1$-bundle over $S$. Indeed, if $x\in S$ and $(u,v)$ is an orthonormal basis of $T_xS$, the map
	\[
	\begin{array}{cccc}
	f:&\S^1&\to&ST_xS\\
	&\theta&\mapsto&u\cos\theta +v\sin\theta
	\end{array}
	\]
	is a diffeomorphism.
	
	It is known \cite[p. 27]{geiges} that $g$ allows us to define a diffeomorphism $\Psi$ from the tangent bundle $TS$ to the cotangent bundle $T^*S$ which is fibrewise given by
	\[
	\begin{array}{cccc}
	\Psi_x:&T_xS&\to&T^*_xS\\
	&u&\mapsto&g(u,-).
	\end{array}
	\]
	Such a diffeomorphism defines a metric $g^*$ on $T^*S$ given by $g^*(\omega_1,\omega_2) = g(\Psi^{-1}(\omega_1), \Psi^{-1}(\omega_2))$, which in turn allows us to define the unit cotangent bundle as
	\[
	ST^*S = \{\omega\in T^*S\ |\ g^*(\omega, \omega) = 1\}.
	\]
	
	The unit cotangent bundle $ST^*S$ carries a canonical contact structure defined as follows, see \cite[Ex. 2.2]{chernov} or \cite{canonicalcontactunit}. Let $\tilde{\pi}:ST^*S\to S$ denote the canonical projection. A point $\omega\in ST^*S$ may be regarded as a linear form $\tilde{\omega}\in T^*_{\tilde{\pi}(\omega)}S$ up to multiplication by a positive scalar. Thus, $\tilde{\omega}$ is totally determined by the cooriented hyperplane $l_\omega = \ker\tilde{\omega}\subset T_{\tilde{\pi}(\omega)}S$, where the coorientation is given by the half-space on which $\tilde{\omega}$ is positive.  The canonical contact distribution on $ST^*S$ is 
	\[
	\xi_\omega = (T_\omega\tilde{\pi})^{-1}(l_\omega).
	\]
	
	Let us show that $\xi_\omega$ is indeed a contact distribution. Let $(x_1,x_2)$ be orthogonal coordinates on an open subset $U\subset S$, which always exist due to the existence of isothermal coordinates \cite{isothermal}. Isothermal coordinates are those for which the metric is pointwise proportional to the Euclidean metric. Let $(\partial_{x_{1}}, \partial_{x_2})$ be the basis of coordinate vectors and $(dx_1, dx_2)$ be its dual basis pointwise. Let $||\partial_{x_i}|| = \sqrt{g(\partial_{x_i}, \partial_{x_i})}$ and $||dx_i|| = \sqrt{g^*(dx_i, dx_i)}$. Note that 
	\[
	g\Big(\partial_{x_1}, \frac{\partial_{x_1}}{||\partial_{x_1}||^2}\Big) = 1 \hspace*{1cm}\text{and}\hspace{1cm}g\Big(\partial_{x_2}, \frac{\partial_{x_1}}{||\partial_{x_1}||^2}\Big) = 0,
	\]
	which implies that $\Psi^{-1}(dx_1) = \frac{\partial_{x_1}}{||\partial_{x_1}||^2}$, and similarly for $\Psi^{-1}(dx_2) = \frac{\partial_{x_2}}{||\partial_{x_2}||^2}$. Hence,
	\[
	||dx_1|| = \sqrt{g\Bigg(\frac{\partial_{x_1}}{||\partial_{x_1}||^2},\frac{\partial_{x_1}}{||\partial_{x_1}||^2}\Bigg)} = \frac{1}{||\partial_{x_1}||},
	\]
	and $||dx_2|| = \frac{1}{||\partial_{x_2}||}$. 
	
	Let now $\theta$ be the coordinate on the fibres of $ST^*S$, that is, the triplet $(x_1,x_2,\theta)$ represents the point $\omega = \cos\theta \frac{dx_1}{||dx_1||}+\sin\theta \frac{dx_2}{||dx_2||}\in ST_{(x_1,x_2)}^*S$. It is clear that $l_{\omega} = \lspan \sin\theta\frac{\partial_{x_1}}{||\partial_{x_1}||}-\cos\theta\frac{\partial_{x_2}}{||\partial_{x_2}||}\rspan\subset T_{(x_1,x_2)}S$, and hence 
	\[
	\xi_\omega = \lspan \sin\theta\frac{\partial_{x_1}}{||\partial_{x_1}||}-\cos\theta\frac{\partial_{x_2}}{||\partial_{x_2}||},\partial_\theta\rspan,
	\]
	where we have made an abuse of notation denoting by $\partial_{x_i} $ the coordinate vector fields on $ST^*S$ as well.
	
	It is now clear that, on $\tilde{\pi}^{-1}(U)$, we can write $\xi = \ker\alpha$ with 
	\[
	\alpha =\cos\theta \frac{dx_1}{||dx_1||}+\sin\theta \frac{dx_2}{||dx_2||},
	\]
	again by making an abuse of notation. Then, 
	\[
	d\alpha = -\frac{\sin\theta}{||dx_1||} d\theta\wedge dx_1+\frac{\cos\theta}{||dx_2||} d\theta\wedge dx_2 + \cos\theta\frac{\partial}{\partial x_2}\frac{1}{||dx_1||}dx_2\wedge dx_1 + \sin\theta\frac{\partial}{\partial x_1}\frac{1}{||dx_2||}dx_1\wedge dx_2
	\]
	 and
	\[
	\alpha\wedge d\alpha = -\frac{\sin^2\theta}{||dx_1||||dx_2||} dx_2\wedge d\theta\wedge dx_1+\frac{\cos^2\theta}{||dx_1||||dx_2||} dx^1\wedge d\theta\wedge dx_2 = -\frac{1}{||dx_1||||dx_2||}dx_1\wedge dx_2\wedge d\theta\neq0,
	\]
	which shows that $\xi$ is indeed a contact structure on $ST^*S$.
	
	Note now that the pushforward
	\[
	\chi = (\Psi^{-1})_* \xi
	\]
	defines a contact structure on $STS$, which we will call the \textbf{canonical contact structure} on $STS$. Let $\pi:STS\to S$ be the canonical projection and $u\in STS$. Then
	\begin{align*}
		\chi_u =& \big((\Psi^{-1})_*\xi\big)_u \\ = &  T_{\Psi(u)}(\Psi^{-1})(\xi_{\Psi(u)}) \\ = & T_{\Psi(u)}(\Psi^{-1})\circ (T_{\Psi(u)}\tilde{\pi})^{-1}(l_{\tilde{\pi}(u)}) \\ = & \big(T_{u}(\Psi\circ\tilde{\pi} )\big)^{-1}\big(\ker g(u,-)\big) \\ = &  (T_u\pi)^{-1}(\lspan u\rspan^\perp),
	\end{align*}
where $\lspan u\rspan^\perp$ denotes the orthogonal subspace to $u$ in $T_{\pi(u)}S$ defined by $g$. 
\end{Example}
\section{Introduction to Lorentzian Manifolds}
\label{IntroToLorentz}
Section \ref{IntroToLorentz} introduces the basics of pseudo-Riemannian manifolds and geodesics. Let $M$ be a connected manifold of dimension $m$.
\begin{Definition}
	A \textbf{pseudo-Riemannian metric} on $M$ is a collection $\{g_x\}_{x\in M}$ of  non-degenerate symmetric bilinear forms on the tangent bundle of $M$, that is, for any $x\in M$, one has
	\[
	g_x:T_xM\times T_xM\to\R
	\]
	such that for all $u,v,v'\in T_xM$ and $a,b\in\R$
	\begin{enumerate}
		\item $g_x(u,v) = g_x(v,u)$,
		\item $g_x(u,av+bv') = ag_x(u,v)+bg_x(u,v')$,
		\item if $g_x(u,v) = 0$ for all $v\in T_xM$, then $u = 0$.
	\end{enumerate}
In addition, it is required that $g_x$ varies smoothly with respect to the base point $x$. That is, for any $X,Y\in\mathfrak{X}(M)$, the function $x\mapsto g_x(X(x), Y(x))$ is smooth. We will simply denote the metric by $g$. The pair $(M, g)$ is called a \textbf{pseudo-Riemannian manifold}.
\end{Definition}

Note that it is not required for the metric to be positive definite. Actually, a Lorentzian manifold will be a manifold equipped with a metric that has one negative direction. The next few results will allow us to formalise this concept.

\begin{Theorem}[Sylvester's law of inertia]
	\cite[Prop. 2.65]{lee} Let $h$ be a non-degenerate symmetric bilinear form on a finite-dimensional vector space $V$. Let $A$ be the matrix representation of $h$ in some basis of $V$. Then, $A$ is diagonalisable and the number of positive and negative eigenvalues of $A$ is independent of the choice of basis. 
\end{Theorem}
Sylvester's law of inertia allows us to define the signature of the metric $g$ in each tangent space of $M$. We will say that the signature of $g$ at $x\in M$ is the pair $(r,s)$, where $r$ is the number of positive eigenvalues of $g_x$ and $s$ is the number of negative eigenvalues.

This last result allows us to define the signature of a metric $g$ on $M$, which will simply be the signature of $g_x$ for any $x\in M$. It can be seen that the signature of the metric is locally constant.

\begin{Lemma}
	The signature of a metric $g$ on $M$ is locally constant.
\end{Lemma}
\begin{proof}
	Let $x\in X$ and $(U,\varphi)$ be a chart of $M$ centred at $x$. Let $\partial_{x_i}$ denote the coordinate vectors for $i = 1,\ldots, m$ defined by $\varphi$. Then, $\big(\partial_{x_1}(y), \ldots, \partial_{x_m}(y)\big)$ provides a basis of $T_yM$ for all $y\in U$, and the matrix representation of $g$ in this basis is
	\[
	G(y)= \big(g_{ij}(y)\big) := \big(g(\partial_{x_i}, \partial_{x_j})\big).
	\]
	
	Thus, the fact that the signature of $g$ is locally constant around $x$ is equivalent to the number of positive eigenvalues of $G(y)$ being constant around $x$. Since $G(y)$ is non-degenerate, it has no null eigenvalues. The eigenvalues of $G(y)$ are the roots of a monic polynomial of degree $m$ whose coefficients $c_i\in\R$ are products and sums of the entries of $G$. Since the functions $g_{ij}$ are smooth, the entries of $G(y)$ depend smoothly on $y$ and so do the coefficients $c_i$. It is known that the roots of a monic polynomial of positive degree depend continuously on its coefficients \cite{roots}, and hence the statement follows.
\end{proof}

We can now give a precise definition of a Lorentzian manifold.
\begin{Definition}
	Let $(M,g)$ be a connected pseudo-Riemannian manifold, where $M$ is of dimension $m$. We say that $(M,g)$ is a \textbf{Lorentzian manifold} if the signature of $g$ is $(m-1,1)$. 
\end{Definition}

We will say that the vectors $u_1,\ldots,u_{m-1},v\in T_xM$ form an orthonormal basis of $T_xM$ if $(u_1,\ldots,u_{m-1},v)$ is a basis of $T_xM$ and, in addition
\begin{enumerate}
	\item $ g(u_i,u_i) = 1$ and $g(u_i,v) = 0$ for all $i = 1,\ldots, m-1$,
	\item $g(u_i, u_j) = 0$ whenever $i\neq j$,
	\item $g(v,v) = -1$.
\end{enumerate}

 From now on until the end of the section, let $(M,g)$ be a Lorentzian manifold.
\begin{Definition}
	Let $x\in M$ and $u\in T_xM$. We say that $u$ is 
	\begin{enumerate}
		\item \textbf{space-like} if $g(u,u)>0$ or $u = 0$,
		\item \textbf{light-like} or \textbf{null} if $g(u,u) = 0$ and $u\neq 0$,
		\item \textbf{time-like} if $g(u,u)<0$,
		\item \textbf{non-space-like} if $g(u,u)\leq 0$ and $u\neq 0$.
	\end{enumerate}

If $c:I\to M$ is a smooth curve, we say that $c$ is space-like, light-like, time-like or non-space-like if $\dot{c}(t)$ is respectively space-like, light-like, time-like or non-space-like for all $t\in I$.
\end{Definition}

Our next goal is to define the concept of geodesic on $(M,g)$ and provide tools to compute them. We need to present the Levi-Civita connection first. 
\begin{Definition}
	An \textbf{affine connection} on a manifold $M$ is a map 
	\[
	\begin{array}{cccc}
	\nabla:& \mathfrak{X}(M)\times\mathfrak{X}(M)&\to&\mathfrak{X}(M)\\
	&(X,Y)&\mapsto&\nabla_XY
	\end{array}
	\]
	such that
	\begin{enumerate}
		\item $\nabla_{fX+Y}Z = f\nabla_XZ+\nabla_YZ$,\\
		\item $\nabla_X(fY+Z) = X(f)Y+f\nabla_XY+\nabla_XZ,$
	\end{enumerate}
for all $X,Y,Z\in\mathfrak{X}(M)$ and $f\in\mathcal{C}^\infty(M)$.
We say that the connection $\nabla$ is \textbf{symmetric} if $\nabla_XY -\nabla_YX = [X,Y]$ for all $X,Y\in \mathfrak{X}(M)$. In addition, if $M$ is equipped with a pseudo-Riemannian metric $g$, we say that $\nabla$ and $g$ are \textbf{compatible} if 
\[
Xg(Y,Z) = g(\nabla_XY,Z)+g(Y,\nabla_XZ).
\]
\end{Definition}

\begin{Theorem}[Fundamental Theorem of Riemannian geometry]
	\cite[Thm. 3.11]{oneill} There exists a unique affine connection $\nabla$ on $(M,g)$ which is symmetric and compatible with $g$. We call $\nabla$ the Levi-Civita connection of $(M,g)$.
\end{Theorem}

If $(U,\varphi)$ is a local chart of $M$, let us denote by $\partial_{x_i}(x)$ the coordinate vector fields induced by the chart. Then, it is known that $\{\partial_{x_i}(x)\}_{i = 1}^m$ is a basis of $T_xM$ for all $x\in U$. Hence, there exist functions $\Gamma_{ij}^k:U\to\mathbb{R}$ such that
\[
\nabla_{\partial_{x_i}}\partial_{x_j} = \sum\limits_{k = 1}^m\Gamma^k_{ij}\partial_{x_k}
\]
in $U$ for $i,j,k = 1,\ldots,m$. The smooth functions $\Gamma^k_{ij}$ are known as the Christoffel symbols of the connection. Recall that, in a local chart, one can also define the smooth functions $g_{ij}:U\to\R$ given by $g_{ij}(x) = g_x(\partial_{x_i}(x), \partial_{x_j}(x))$. It can be shown \cite[Prop. 3.13.2]{oneill} that the Christoffel symbols of the Levi-Civita connection are given, in a local chart of $M$, by
\begin{equation}
	\label{Christoffel}
\Gamma_{ij}^k = \frac{1}{2}\sum\limits_{n= 1}^mg^{kn}\Big(\frac{\partial g_{jn}}{\partial x_i}+\frac{\partial g_{in}}{\partial x_j}-\frac{\partial g_{ij}}{\partial x_n}\Big)
\end{equation}
where $(g^{ij})$ is the inverse matrix of $(g_{ij})$, and these completely determine $\nabla$.

The Levi-Civita connection allows us to define the covariant derivative of a vector field on a curve as follows. Let $\gamma:I\to M$ be a smooth curve on $M$. A vector field on $\gamma$ is a map $V:I\to TM$ such that $V(t)\in T_{\gamma(t)}M$ for all $t\in I$. It can be shown \cite[Rk. 2.2.2.3]{docarmo} that the value of $\nabla_XY(x)$ depends solely on $X(x)$ and the values of $Y$ on a curve tangent to $X(x)$ at $x$, for all $x\in M$ and $X,Y\in\mathfrak{X}(M)$. Hence, one can define the covariant derivative of $V$ as the unique vector field on $\gamma$ which is given by
\begin{equation}
	\label{CovariantDerivative}
\frac{D}{dt}V := \nabla_{\dot{\gamma}}V.
\end{equation}

In a local chart $(U,\varphi)$, the covariant derivative operator reads \cite[p. 66]{oneill}
\begin{equation}
	\label{CovariantDerivativeChart}
	\frac{D}{dt}V = \sum\limits_{i = 1}^m\frac{V_i}{dt}\partial_{x_i}+\sum\limits_{i,j,k = 1}^m V_j\dot{x}_i\Gamma_{ij}^k\partial_{x_k},
\end{equation}
if $V = \sum\limits_{i = 1}^mV_i\partial_{x_i}$ and $\varphi\circ\gamma(t) = \big(x_1(t),\ldots, x_m(t)\big)$. The notion of covariant derivative allows us to define the geodesics on $(M,g)$.
\begin{Definition}
	A parametrized curve $\gamma: I\to M$ is a \textbf{geodesic} of $(M,g)$ if 
	\[
	\frac{D}{dt}\dot{\gamma} = \nabla_{\dot{\gamma}}\dot{\gamma}\equiv 0.
	\]
\end{Definition}
Equation \eqref{CovariantDerivativeChart} allows us to express the geodesic condition in a chart $(U,\varphi)$. A curve $\gamma: I\to U$ given by $\varphi\circ\gamma(t) = \big(x_1(t), \ldots, x_m(t)\big)$ is a geodesic if and only if
\begin{equation}
	\label{GeodesicChart}
	\ddot{x}_k + \sum\limits_{i,j = 1}^m\dot{x}_i\dot{x}_j\Gamma^k_{ij} \equiv 0,
\end{equation}
for all $k = 1,\ldots, m$. It can be shown \cite[p. 61]{docarmo} that all geodesics $\gamma: I \to M$ have constant length, i.e.
\begin{equation}
	\label{constant}
\frac{d}{dt}g(\dot{\gamma}, \dot{\gamma})\equiv 0.
\end{equation}

We finish this section with the following result on existence and uniqueness of geodesics.

\begin{Proposition}
	\label{EandUgeodesics}
	\cite[Prop. 3.22]{oneill} Let $x\in M$ and $u\in T_xM$. There exists an interval $I$ about 0 and a unique geodesic $\gamma: I \rightarrow M$ of $M$ such that $\gamma(0) = x$ and $\dot{\gamma}(0) = u$.
\end{Proposition}

\section{The Space of Null Geodesics of a Spacetime}
\label{SpaceOfNullG}
In this section we introduce the concept of spacetime and that of its space of null geodesics. We also present the canonical contact structure of such space.

Consider a Lorentzian manifold $(M, g)$.  We say that $(M,g)$ is \textbf{time-orientable} if there exists a time-like vector field $X\in\mathfrak{X}(M)$. That is, $g_x(X(x), X(x))<0$ for all $x\in M$. A choice of such a vector field is called a choice of future within $M$.

\begin{Proposition}
	\label{Propcones}
	Let $(M, g)$ be a Lorentzian manifold. For all $x\in M$, the set of non-space-like vectors in $T_xM$ forms a solid cone consisting in two solid hemicones, the boundary of which is precisely the set of all null vectors in $T_xM$. In addition, the cone varies smoothly with respect to the basepoint $x$.
\end{Proposition}
\begin{proof}
	Let $x\in M$ and let $(u_1,\ldots u_{m-1}, v)$ be an orthonormal basis of $T_xM$. Let $\mathcal{I}C_x$ denote the set of non-space-like vectors in $T_xM$.
	
	Let now $\lambda_1,\ldots,\lambda_{m-1},\mu\in \R$ not all zero. Then,
	\begin{align*}
		\sum\limits_{i = 1}^{m-1}\lambda_iu_i +\mu v\in \mathcal{I}C_x&\iff  g_x\big(\sum\limits_{i = 1}^{m-1}\lambda_iu_i +\mu v, \sum\limits_{i = 1}^{m-1}\lambda_iu_i +\mu v\big)\leq 0 \iff \\ &\iff \sum\limits_{i = 1}^{m-1}\lambda_i^2-\mu^2\leq 0,
	\end{align*}
	and the claim follows. It is clear that the boundary $C_x$ of the cone is precisely formed by the null vectors in $T_xM$. The differentiable fashion in which the cone depends on $x$ is clear from the fact that $g_x$ varies smoothly with respect to $x$.  
\end{proof}

\begin{Definition}
	A \textbf{spacetime} is a time-orientable connected Lorentzian manifold of dimension $m\geq 3$. 
\end{Definition}

From now on, let $(M,g)$ be a spacetime. We will denote by $C$ the bundle of null vectors, which is pointwise the boundary of the bundle of solid cones of non-space-like vectors $\mathcal{I}C$.

Proposition \ref{Propcones} ensures that $C$ inherits a structure of manifold from $TM$. Now, a choice of future is simply a differentiable choice of one of such hemicones on every tangent space. We will denote by $C^+$ the bundle of null future vectors of $M$, and by $\mathcal{I}C^+$ the bundle of future-pointing non-space-like vectors. A curve $\gamma:I\rightarrow M$ is said to be future-pointing if $\dot{\gamma}(t)\in\mathcal{I}C^+_{\gamma(t)}$ for all $t\in I$. We can now define the set of null geodesics of $M$.

\begin{Definition}
	The set of unparametrized future-pointing null geodesics of a spacetime $(M,g)$, or simply of \textbf{the set of null geodesics}, is
	\[
	\mathcal{N} = \{\gamma(I)\ |\ \gamma:I\rightarrow M\text{ is a maximal future-pointing null geodesic in } (M,g)\}.
	\]
\end{Definition}

Our next goal is to show that $\mathcal{N}$ can be defined as a quotient of the bundle $C^+$ and that it can be given structure of a manifold under mild assumptions on $M$. Let us first define the geodesic spray $X_g\in\mathfrak{X}(TM)$.

\begin{Definition}
	The \textbf{geodesic spray} $X_g\in\mathfrak{X}(TM)$ is the vector field on the tangent bundle of $M$ whose integral lines are the curves $\dot{\gamma}(t)\in T_{\gamma(t)}M$, where $\gamma: I\rightarrow M$ is a geodesic of $M$. 
\end{Definition}

It can be seen that the geodesic spray is tangent to the bundle $C$. Indeed, let us define $f:TM\rightarrow \R$ given by $f(u) = g(u,u)$. Now, for any geodesic $\gamma: I\rightarrow M$, one has that $\dot{\gamma}$ is an integral line of $X_g$ and $f(\dot{\gamma}) = g(\dot{\gamma}, \dot{\gamma})$ is constant, by Equation \eqref{constant}. For any $u\in T_xM$, the integral line of $X_g$ going through $u$ at time zero is $\dot{\gamma}(s)$, where $\gamma(s)$ is a geodesic in $M$ such that $\gamma(0) = x$ and $\dot{\gamma}(0) = u$. Then,
\[
T_u f(X_g(u)) = \frac{d}{ds}\big{|}_{s = 0}f(\dot{\gamma}(s)) = 0,
\]  
and $X_g$ is tangent to any level set of $f$ and, in particular, to $C = f^{-1}(0)$. Moreover, if $\dot{\gamma}(s)$ is the integral line of $X_g$ through $u\in C^+_x$ at time zero, then $\gamma(0) = x$ and $\dot{\gamma}(0) = u$. In addition, $\dot{\gamma}(s)\in C^+$ for all $s\in(-\varepsilon, \varepsilon)$ for some $\varepsilon>0$. Hence, the geodesic spray $X_g$ is tangent to $C^+$.

Let us also define the Euler vector field $\Delta\in\mathfrak{X}(TM)$ as
\[
\Delta(u) = T_0c(\partial_s),
\]
where $u\in T_xM$ and $c:\R\rightarrow T_xM$ is given by $c(s) = e^{s}u$. Note that the Euler field simply dilates vectors in $T_xM$. Note that 
\[
\dot{c}(s) = T_s c(\partial_s) = \Delta(c(s)),
\]
and hence $c$ is an integral line of $\Delta$. If $u\in C^+$, then $c$ entirely lies in $C^+$, which shows that the Euler field $\Delta$ is tangent to $C^+$. 

We have showed that both fields $X_g, \Delta\in\mathfrak{X}(TM)$ are vector fields on $C^+$. Let us give some intuition on why these vector fields are useful for our argument. Firstly, quotienting $C^+$ by $\Delta$ is going to projectivise the cones, which is relevant because in $\mathcal{N}$ we only care of unparametrized geodesics. Also, intuitively, quotienting by the geodesic spray $X_g$ will identify those vectors that define the same geodesic, that is, those vectors in different cones for which there exists a geodesic going through both of them. This intuitive idea is formalised by the fact that the integral lines of the geodesic spray are precisely the geodesics in $M$ together with their tangent vectors. In fact, by quotienting $C^+$ by $\Delta$ and $X_g$ one obtains the manifold of future-pointing unparametrized null geodesics $\mathcal{N}$. The only thing left to show is that the distribution
\[
\mathcal{D} = \lspan \Delta, X_g\rspan
\]
is integrable on $C^+$. A straightforward  computation in coordinates \cite[p. 15]{bautista} shows that
\[
[X_g,\Delta] = X_g\in\mathcal{D},
\]
and hence, by Frobenius' Theorem \ref{Frobenius}, the distribution $\mathcal{D}$ is integrable and
\[
\mathcal{N} = C^+/\mathcal{D}.
\]

The previous description of the space of null geodesics allows us to formulate the following result, whose proof can be found in \cite[Sec. 2.2]{bautista}.
\begin{Theorem}
	\label{CharactManifolds}
	Let $(M,g)$ be a spacetime such that
	\begin{enumerate}
		\item for all $x\in M$ and every neighbourhood $U\subset M$ of $x$, there exists a neighbourhood $V\subset U$ of $x$ such that any segment of non-space-like curve with endpoints in $V$ is wholly contained in $U$,
		\item for any compact $K\subset M$, there exists a compact $K'\subset M$ such that any null geodesic segment with endpoints in $K$ is totally contained in $K'$.
	\end{enumerate}
Then, $\mathcal{N}$ inherits the structure of a smooth manifold from that of $C^+$. In addition, such structure is the only one for which the canonical projection
\[
\pi_\mathcal{N}:C^+\rightarrow\mathcal{N} = C^+/\mathcal{D}
\]
is a submersion.
\end{Theorem}

\begin{Example}
	\label{Minkowski}
	Let us consider the three-dimensional Minkowski space $\mathbb{M}^3 = (\mathbb{R}^3, g)$, where 
	\[
	g = dx^2+dy^2-dz^2.
	\]
	
	Let $\partial_z\in\mathfrak{X}(\R^3)$ be a choice of future. It is clear that $\mathbb{M}$ satisfies condition $i)$ in Theorem \ref{CharactManifolds}. The matrix representation of $g$ in the global chart $(\R^3, id)$ is
	\[
	G = \text{\normalfont diag}(1,1,-1),
	\]
	and, using Equation \eqref{Christoffel}, it is easy to see that all the Christoffel symbols vanish. Hence, the geodesic equation implies that  $\gamma:\R\rightarrow\mathbb{M}^3$ is a geodesic if and only if
	\[
	\gamma(t) = \big(v_xt+x_0, v_yt+y_0, v_zt+z_0\big),
	\]
	where $x_0,y_0,z_0,v_x,v_y,v_z\in\R$. 
	
	Let now $K\subset \mathbb{M}^3$ be any compact subset of $\mathbb{M}^3$, then $K$ is bounded and thus there exists a closed ball $\overline{B}$ that contains it. Since $\overline{B}$ is convex, any geodesic segment with endpoints in $K\subset\overline{B}$ is wholly contained in $\overline{B}$. Hence, all hypotheses of Theorem \ref{CharactManifolds} are satisfied and $\mathcal{N}= C^+/\mathcal{D}$ is a smooth manifold.
	
	It is clear that $\gamma$ is a null geodesic if and only if
	\[
	0 = g(\dot{\gamma}, \dot{\gamma}) = g\big((v_x,v_y,v_z), (v_x,v_y,v_z)\big) = v_x^2+v_y^2-v_z^2.
	\]
	
	Imposing that $\gamma$ is future-pointing is equivalent to $v_z>0$, and hence we can write
	\[
	v_z = \sqrt{v_x^2+v_y^2}.
	\]
	
	One can always reparametrize $\gamma$ so that $z_0 = 0$ and $v_z = 1= v_x^2+v_y^2$. Then,
	\[
	\gamma(t) = \big(v_x^2 t+x_0, v_y^2 t+y_0, t\big),
	\]
	and the geodesic is completely determined by $(x_0,y_0)\in\{z = 0\}\cong\R^2$ and $(v_x,v_y)\in ST_{(x_0,y_0)}\{z = 0\}\cong ST_{(x_0,y_0)}\R^2$. It is also clear that any two such pairs will determine a unique unparametrized future pointing null geodesic and, hence, $\mathcal{N} \cong ST\R^2$.
\end{Example}

We now present the canonical contact structure on the space of null geodesics $\mathcal{N}$. Let us introduce the concept of sky.

\begin{Definition}
	Let $x\in M$. The \textbf{sky} of $x$ is
	\[
	\mathfrak{S}_x=\{\gamma\in\mathcal{N}\ |\ x\in\gamma\subset M\}\subset\mathcal{N},
	\]
	the set of geodesics that contain $x$.
\end{Definition}

Note that, for any $x\in M$, the sky of $x$ is in correspondence with the projectivisation of $C^+_x$, since any one of such projectivised vectors defines an unparametrized null geodesic that contains $x$, and any null geodesic through $x$ is tangent to a line in $C^+_x$. Hence, $\mathfrak{S}_x\cong \S^{m-2}$, where $m$ is the dimension of $M$.

The contact structure on $\mathcal{N}$ is defined in terms of the skies of the points of $M$.

\begin{Definition}
	The \textbf{canonical contact structure} of the space of null geodesics $\mathcal{N}$ is the field of hyperplanes $\mathcal{H}$ on $\mathcal{N}$ defined pointwise as follows. Let $\gamma\in\mathcal{N}$ and let $x,y\in\gamma\subset M$ be close enough such that $T_\gamma\mathfrak{S}_x\cap T_\gamma\mathfrak{S}_y = \{0\}$. Then,
	\[
	\mathcal{H}_\gamma = T_\gamma\mathfrak{S}_x\oplus T_\gamma\mathfrak{S}_y.
	\]
\end{Definition}

For $\mathcal{H}$ to be well defined it is necessary to show that one can always find such two points $x,y\in\gamma$. In addition, one must show that the direct sum $T_\gamma \mathfrak{S}_x\oplus T_\gamma\mathfrak{S}_y$ is independent of $x$ and $y$ and that it varies smoothly with respect to $\gamma$. It is also necessary to show that $\mathcal{H}$ is indeed a contact structure on $\mathcal{N}$. All of these verifications are beyond the scope of this work, so we refer to \cite[Sec 2.4]{bautista}. However, it is worth noting that for any $x\in\gamma$, the sky of $x$ is tangent to the contact distribution at $\gamma$.
\section{Engel Structures and Prolongations}
\label{EngelSection}
We finally move on to Engel structures and their relation to contact manifolds and Lorentzian manifolds via prolongation and deprolongation maps, which will be central in Chapter \ref{ProlandDeprolChap}.

\begin{Definition}
	Let $M$ be a manifold and $\xi_1,\xi_2\subseteq TM$ distributions on $M$. We define the distribution
	\[
	[\xi_1,\xi_2] = \bigsqcup\limits_{x\in M}\{[X,Y]_x\ |\ X\in \xi_1,\ Y\in \xi_2\}.
	\]
\end{Definition}

\begin{Definition}
	Let $M$ be a four-manifold. A rank-three distribution $\mathcal{E}\subset TM$ on $M$ is said to be an \textbf{even-contact structure} if it is everywhere non-integrable, i.e. if $[\mathcal{E},\mathcal{E}] = TM$.
\end{Definition}
\begin{Definition}
	Let $M$ be a four-manifold. A rank-two distribution $\mathcal{D}\subset TM$ on $M$ is an \textbf{Engel structure} (or simply Engel) if $\mathcal{E} = [\mathcal{D},\mathcal{D}]$ is an even-contact structure on $M$. We will say that $\mathcal{D}$ is an Engel structure (or simply Engel) at $x\in M$ if there exists a neighbourhood $U$ of $x$ in $M$ such that $\mathcal{D}|_U$ is an Engel structure on $U$.
\end{Definition}

Let us now present some well-known results on Engel structures.
\begin{Proposition}
	Let $M$ be a four-manifold. Then,
	\begin{enumerate}
		\item if $\mathcal{E}$ is an even-contact structure on $M$, there exists a unique line field $\mathcal{W}\subset\mathcal{E}$ defined by $[\mathcal{W}, \mathcal{E}] \subseteq\mathcal{E}$. The line field $\mathcal{W}$ is called the \textbf{kernel} (or characteristic line field) of the even-contact distribution.
		
		\item if $\mathcal{D}$ is an Engel structure on $M$ and $\mathcal{E} = [\mathcal{D}, \mathcal{D}]$, it holds that $\mathcal{W}\subset\mathcal{D}$.
	\end{enumerate}
\end{Proposition}
\begin{proof}
	Let $M$ be a four-manifold and $\mathcal{E}$ an even-contact structure on $M$. Since $\mathcal{E}$ is a field of hyperplanes on $M$, the distribution $\mathcal{E}$ can be represented, at least locally, as the kernel of a one-form $\theta\in\Omega^1(U)$, for a neighbourhood $U\subset M$, as stated in Lemma \ref{kerneloneform}. Then, the non-integrability condition of $\mathcal{E}$ is equivalent to $d\theta|_{\mathcal{E}_y}$ being a two-form of maximal rank for any $y\in U$, see \cite[p. 3]{geiges}. Since $\mathcal{E}_y$ is of odd dimension, the two-form $d\theta|_{\mathcal{E}_y}$ will have a non-zero kernel. Since the two-form is of maximal rank, the kernel will be of minimal dimension, that is, dimension one. These kernels define the line field $\mathcal{W}_y\subset\mathcal{E}_y$.
	
	It is clear that, if $X\in \mathcal{W}$ and $Y\in\mathcal{E}$, one has
	\[
	\theta([X,Y]) = X(\theta(Y))-Y(\theta(X))-d\theta(X,Y) = 0, 
	\]
	which implies that $[X,Y]\in\mathcal{E}$, and $[\mathcal{W},\mathcal{E}]\subseteq\mathcal{E}$. If $\mathcal{W}'$  is another line field in $\mathcal{E}$ satisfying $[\mathcal{W}',\mathcal{E}]\subseteq\mathcal{E}$, by the above expression, any vector field in $\mathcal{W}'$ must lie in the kernel of $d\theta$ necessarily.
	
	Now, let $\mathcal{D}$ be an Engel structure on $M$ such that $[\mathcal{D},\mathcal{D}] = \mathcal{E}$. Let $\mathcal{W}$ be the kernel of $\mathcal{E}$. Assume $\mathcal{W}\subset\mathcal{D}$ does not hold. Then, there exists $x\in M$ for which $\mathcal{W}_x\nsubseteq\mathcal{D}_x$. Then, necessarily, $\mathcal{W}_y\oplus\mathcal{D}_y = \mathcal{E}_y$ for all $y$ in a neighbourhood of $x$ and therefore $[\mathcal{E},\mathcal{E}]_x = [\mathcal{D},\mathcal{D}]_x = \mathcal{E}_x$, which is a contradiction. Hence, the kernel does lie within $\mathcal{D}$.
\end{proof}

Let $B^3 = \{p\in\R^3\ |\ ||p||<1\}$ be the three-dimensional ball and consider the four-manifold $B^3\times [0,1]$ with coordinates $(x,y,z,t)$. Let $X\in\mathfrak{X}(B^3\times[0,1])$ be such that $X(x,y,z,t)\in T_{(x,y,z,t)}(B^3\times \{t\})$ and define $\mathcal{D} = \lspan \partial_t,X\rspan$. Let us define $\dot{X} = [\partial_t,X]$ and $\ddot{X} = [\partial_t,\dot{X}]$, which are also tangent to the level sets $B^3\times\{t\}$ for all $t\in[0,1]$. Hence, the three vector fields $X,\dot{X},\ddot{X}$ can be regarded as uniparametric families of vector fields on $B^3$ with parameter $t\in[0,1]$. Let us denote such families by $X_t,\dot{X}_t,\ddot{X}_t$.

Since $( \partial_x,\partial_y,\partial_z)$ gives a basis of $T_p B^3$ for all $p\in B^3$, all tangent spaces of $B^3$ can be identified, as well as all the fibres of the unit tangent bundle $S TB^3$. Hence, for $p\in B^3$, the map $t\mapsto X_t(p)$  describes a curve in $\S^2$, which allows us to see the distribution $\mathcal{D}$ as a $B^3$-family of curves on $\S^2$. See \cite{delpino, frannonint, existenceh} for more details.

\begin{Definition}
	Let $\gamma:[0,1]\rightarrow\S^2$ be a smooth curve such that $\gamma'(t)\neq 0$ for all $t\in [0,1]$. Let $\mathfrak{t}(t)= \frac{\gamma'(t)}{||\gamma'(t)||}$ and $\mathfrak{n}(t)$ be the unique vector field such that $(\mathfrak{t}(t),\mathfrak{n}(t))$ is an orthonormal oriented basis of $T_{\gamma(t)}\S^2$. 
	
	We say a point $\gamma(t)$ is an \textbf{inflection point} of $\gamma$ if $\langle\mathfrak{t}'(t),\mathfrak{n}(t)\rangle = 0$. We say that $\gamma$ is \textbf{convex} if it has no inflection points.
\end{Definition}
\begin{Theorem} \cite[Prop. 8]{existence h}.
	\label{Engelchar}
	Following with the previous notation, a rank-two distribution $\mathcal{D} = \lspan\partial_t,X\rspan$ on $B^3\times[0,1]$ is an Engel structure at $(p,t)\in B^3\times[0,1]$ if both $\dot{X}(p,t)\neq 0$ and one of the following two conditions hold:
	\begin{enumerate}
		\item the curve $X_p:[0,1]\rightarrow\S^2$ has no inflection point at $t$,
		\item the rank-two distribution $\lspan\dot{X}_t,\ddot{X}_t\rspan$ is a contact structure on $N\times\{t\}$ for some neighbourhood $N\subseteq B^3$ of $p$.
	\end{enumerate}
\end{Theorem}

Theorem \ref{Engelchar} allows us to define two types of Engel manifolds. We first present the Cartan prolongation of a contact three-dimensional manifold.

\begin{Example}[The Cartan prolongation]
	\label{Cartanprol}
	Let $(M,\xi)$ be a contact three-manifold and consider the bundle
	\[
	\begin{tikzcd}
		\S^1\arrow[hookrightarrow]{r} &S(\xi) \arrow{d}{\pi_C}\\&M
	\end{tikzcd}
	\]
	where we define $S(\xi)_x$ as the quotient of $\xi_x-\{0\}$ by the equivalence relation $v\sim\lambda v$ for all $\lambda\in\R^+$. Then, $S(\xi)$ carries a canonical Engel structure which is defined as follows. A point in $S(\xi)$ is a pair $(x,L)$ with $x\in M$ and $L\in S(\xi)_x$. Hence, $L$ can be identified with an oriented line in $\xi_x$. The Engel structure is given by
	\[
	\mathcal{D}_{(x,L)} = (T_{(x,L)}\pi_C)^{-1}(L).
	\]
	
	Indeed, let $x\in M$ and $N\subseteq M$ be a chart of $M$ centred at $x$ and diffeomorphic to $B^3$. Let $Y,Z\in\mathfrak{X}(N)$ such that $\lspan Y(y),Z(y)\rspan= \xi_y$ for all $y\in N$. Since $N$ is contractible, $S(\xi)|_N = N\times\S^1$, and we can parametrize $\big(x,L = \lspan Y(x)\cos t+Z(x)\sin t\rspan\big)$ with $t\in\S^1$ on $S(\xi)|_N$. Then, the described rank-two distribution reads
	\[
	\mathcal{D}_{(x,L(t))} = \lspan\partial_t, X(t) = Y\cos t+Z\sin t\rspan.
	\]

	Let $t\in\S^1$. Since $\dot{X}(x,t) \neq0$ and condition $ii)$ of Theorem \ref{Engelchar} are satisfied, $\mathcal{D}$ is Engel at $(x,t)$. Since $(x,t)$ is arbitrary, $\mathcal{D}$ is Engel on $S(\xi)$.
	
	Note that $\mathcal{E} = \lspan\partial_t,X, \dot{X}\rspan = \lspan\partial_t\rspan\oplus\xi$, and, since $[\partial_t,\ddot{X}] = -\dot{X}\in\mathcal{E}$, we have $\mathcal{W} = \lspan\partial_t\rspan$. 
	
	The line field $\mathcal{W}$ is said to be nice if the topological space $M/\mathcal{W}$ formed by the integral lines of $\mathcal{W}$ is a manifold and the canonical projection is a submersion. If we assume $\mathcal{W}$ is nice, then $M \cong S(\xi)/\lspan\partial t\rspan = S(\xi)/\mathcal{W}$. Also, if we denote by $p:S(\xi)\to S(\xi)/\mathcal{W}$ the canonical projection, the contact structure $\xi$ on $M$ is identified with
	\[
	\xi= p_*\mathcal{E}.
	\]
\end{Example}

A similar construction allows us to define an Engel manifold coming from a three-dimensional Lorentzian manifold.

\begin{Example}[The Lorentz prolongation]
	\label{Lorentzprol}
	Let $L$ be a Lorentzian three-manifold. The set of null vectors on $L$ defines the subbundle of cones of $TL$, $\pi|_C:C\rightarrow L$, which induces a bundle
	\[
	\begin{tikzcd}
		\S^1\arrow[hookrightarrow]{r} &\mathcal{P}C\arrow{d}{\pi_L}\\&L
	\end{tikzcd}
	\]
	where $\mathcal{P}C$ is fibrewise the projectivisation of the cone $C$. Then, $\mathcal{P} C$ carries an Engel structure defined as follows. A point $(x,s)\in\mathcal{P} C$ consists of a point $x\in L$ and a line $s$ in $C_x$. Then, let us define 
	\[
	\mathcal{D}_{(x,s)} = (T_{(x,s)}\pi_L)^{-1}(s).
	\]
	
	We will now show this distribution is indeed Engel. Let $x\in L$ and $N\subseteq L$ a chart centred at $x$ and diffeomorphic to $B^3$ and take $(V,Y,Z)$ an orthonormal frame of $TM|_N$. Since $N$ is contractible, $\mathcal{P} C|_N = N\times \S^1$. Let $\theta$ be the coordinate on the fibre $\S^1$. Then, the above defined distribution reads
	\[
	\mathcal{D}_{(x,s(\theta))} = 
\lspan\partial_\theta, X(\theta)= \cos\theta V+\sin\theta Y+Z\rspan.
	\]
	
	If we let the dot denote derivation by $\theta$, it is clear that $\dot{X}_x(\theta) = -\sin\theta V+\cos\theta Y$ is non-zero. Also, the curve $X_x$ is the intersection between the cone of future null vectors at $x$ and the unit sphere of $T_xL$ defined as $\{aV+bY+cZ\in T_xL\ |\ a^2+b^2+c^2 = 1\}$. Hence, it is convex. Thus, since condition $i)$ of Theorem \ref{Engelchar} is satisfied, $\mathcal{D}$ is Engel. Note that the line field $\partial_\theta$ is always transverse to $\mathcal{W}$, since $[\partial_\theta, \dot{X}]\notin\mathcal{E} = \lspan\partial_\theta, X,\dot{X}\rspan$.

\end{Example}

\chapter{The model $\S^2\times\S^1$}
\label{Model}
\section{The Unit Tangent Bundle of $\S^2$}
\label{UTB}
Consider the unit sphere $\S^2\subset\R^3$ given by
\[
\S^2 = \{x\in\R^3\ |\ \langle x, x \rangle = 1\},
\]
where $\langle -,-\rangle$ denotes the Euclidean metric in $\R^3$. Let $\iota:\S^2\to\R^3$ be the canonical inclusion and define the standard metric on $\S^2$ by $g_\circ= \iota^*\langle -, -\rangle$.

Let 
\[
\begin{array}{cccc}
	F:&\R^3&\to &\R \\
	&x&\mapsto & \langle x,x\rangle,
\end{array}
\]
which can be shown to be a submersion. Then, $\S^2 =  F^{-1}(1)$ and hence 
\[
T_x\S^2 = \ker(T_xF) = \{u\in T_x\R^3\ |\ \langle x, u\rangle  = 0\} 
\]
for all $x\in \S^2$. We define the unit tangent bundle of $\S^2$ at $x\in \S^2$ as
\[
ST_x\S^2 = \{u\in T_x\S^2\ |\ \langle u,u\rangle  = 1\} \cong \{u\in\S^2\ |\ \langle x,u \rangle  = 0\},
\]
and the unit tangent bundle of $\S^2$ as $ST\S^2 = \bigsqcup\limits_{x\in\S^2}ST_x\S^2$. Note that we can identify
\[
ST\S^2\cong \{(x,u)\in\S^2\times\S^2\ |\ \langle x, u\rangle = 0\}.
\] 

From Example \ref{ExampleUTB}, the unit tangent bundle $ST\S^2$ inherits the structure of $\S^1$-bundle from $T\S^2$,
	\[
\begin{tikzcd}
	\S^1 \arrow[hookrightarrow]{r} &S T\S^2 \arrow{d}{\pi}\\
	&\S^2
\end{tikzcd}
\]
and carries a canonical contact structure $\chi$ given, for $u\in ST\S^2$, by
\[
\chi_u = (T_u\pi)^{-1}(\lspan  u\rspan^\perp).
\]

It can be shown that the geodesics of $(\S^2, g_\circ)$ are curves of constant speed whose image is a great circle in $\S^2$ \cite[Prop. 5.27]{lee}. Hence, the unique, modulus reparametrisation, geodesic on $(\S^2, g_c)$ going through $x\in \S^2$ with tangent vector $u\in ST_x\S^2$ is
\[
\gamma(t) = x\cos t+u\sin t,
\]
where we make use of the identification $ST\S^2\subset\S^2\times\S^2$. This description of the geodesics of $(\S^2, g_\circ)$ will be repeatedly used in the following sections.
\section{The Lorentzian Manifold $\S^2\times\S^1$}
\label{LorentzianS2S1}
Let $(\S^2, g_\circ)$ be the unit sphere in $\R^3$ equipped with its standard Riemannian metric. Consider the manifold $M = \S^2\times\S^1$ and let $\pi_{\S^1}:M\to\S^1$ be the projection onto the second factor. Let $\iota':\S^1\to\R^2$ be the inclusion and $\langle - , -\rangle_2$ be the Euclidean metric on $\R^2$. Define the pseudo-Riemannian metric 
\[
g_c = g_\circ - \frac{1}{c^2}(\iota'\circ\pi_{\S^1})^*\langle -,-\rangle_2
\] 
on $M$, for $c\in \mathbb{N}^+$. Note that if we let $t$ be the angle coordinate on $\S^1$, the metric $g_c$ reads $g _c= g_\circ -\frac{1}{c^2}dt^2$. 

The pair $(M,g)$ is a Lorentzian manifold in which $\S^2\times\{t\}$ is a space-like surface for all $t\in\S^1$, that is, the metric $g$ restricted to the tangent space of $\S^2\times\{t\}$ is positive-definite. In addition, for all $(x,t)\in \S^2$, the eigenspace of $g_{(x,t)}$ associated with the unique negative eigenvalue of the metric is $T_{(x,t)}(\{x\}\times\S^1 )$. One can take the vector field $(0,\partial_t)\in T(\S^2\times\S^1)\cong T\S^2\times T\S^1$ as a choice of future, and hence $M$ is a spacetime. Similarly to Example \ref{Minkowski}, it is clear that $M$ satisfies condition $i)$ of Theorem \ref{CharactManifolds}. In addition, since $M$ is the cartesian product of two compact manifolds, it follows that $M$ is compact and hence condition $ii)$ of Theorem \ref{CharactManifolds} is satisfied. Hence, the space of null geodesics of $(M,g_c)$ is a differentiable manifold.

Consider $(x,t)\in M$ and $(U, \varphi)$ a local chart of $\S^2$ centred at $x$. Let us also parametrize $\S^1$ around $t$ by the angle coordinate. The cartesian product of both charts provides a chart for $M$ over an open subset $V$ of $M$ around $(x,t)$, for which the metric $g$ has the following matrix representation,
\[(g_c)_{ij}= 
\left(\begin{array}{ccc}
	&         &     0 \\
	\multicolumn{2}{c}{\smash{\raisebox{.5\normalbaselineskip}{$(g_\circ)_{ij}$}}}
	&    0 \\
	\\[-\normalbaselineskip]
	0 &  0&     -\frac{1}{c^2}
\end{array}\right),
\]
where $(g_\circ)_{ij}$ is the matrix representation of $g_\circ$ in the local chart $(U,\varphi)$. The inverse matrix of  $(g_c)_{ij}$ is clearly
\[(g_c)^{ij}= 
\left(\begin{array}{ccc}
	&         &     0 \\
	\multicolumn{2}{c}{\smash{\raisebox{.5\normalbaselineskip}{$(g_\circ)^{ij}$}}}
	&    0 \\
	\\[-\normalbaselineskip]
	0 &  0&      -c^2
\end{array}\right).
\]

It follows from Equation \eqref{Christoffel} that the Christoffel symbol $\Gamma_{ij}^k$ vanishes whenever $i$, $j$ or $k$ equals 3, and that all the others are exactly the Christoffel symbols of $g_\circ$ in the local chart $(U,\varphi)$. Hence, a curve $\gamma:I\to  V $ given by $\gamma(s) = \big(\varphi(y_1(s), y_2(s)), t(s)\big)$ is a geodesic if and only if
\[
\begin{cases}
	\ddot{y}_1 + \Gamma^1_{11}\dot{y}_1^2 + \Gamma^1_{12}\dot{y}_1\dot{y}_2 + \Gamma^1_{22}\dot{y}_2^2 =0\\
	\ddot{y}_2 + \Gamma^2_{11}\dot{y}_1^2 + \Gamma^2_{12}\dot{y}_1\dot{y}_2 + \Gamma^2_{22}\dot{y}_2^2 = 0\\
	\ddot{t} = 0,
\end{cases}
\]
that is, if and only if $t(s) = a+bs$ for some $a, b\in\R$ and $\varphi(y_1(s), y_2(s))$ is a geodesic in $\S^2$. Let $u(s)\in T\S^2$ be the vector tangent to the curve $\varphi(y_1(s), y_2(s))$. Since $\S^2\times\{t\}$ is a space-like surface for all $t\in\S^1$, we can suppose, by reparametrising $\gamma$, that 
\[
g_c\big((u,0), (u,0)\big) = g_\circ(u,u) = 1. 
\]

Then, we find 
\[
g_c(\dot{\gamma}, \dot{\gamma}) = g_\circ(u,u)-\frac{b^2}{c^2}\langle \partial_t, \partial_t\rangle_2 = 1-\frac{b^2}{c^2},
\]
and it follows that $\gamma$ is a future-pointing null geodesic if and only if $b = c$. By uniqueness in Theorem \ref{EandUgeodesics}, we can state that globally, all the null geodesics of $(M, g_c)$ modulus reparametrisation are of the form
\[
\gamma(s) = \big(\mu(s), a+cs\big),
\]
where $\mu$ is a unit-speed great circle in $\S^2$. Note that $\gamma$ intersects $\S^2\times\{0\}$ at least at one point, which implies that, by reparametrising, we can suppose $a = 0$.

Thus, the space of null geodesics $\mathcal{N}_c$ of $(M,g_c)$ defined in Section \ref{SpaceOfNullG} is
\[
\mathcal{N}_c = \{(\mu(s), cs)\ |\ \mu \text{ is a unit-speed great circle in }\S^2\}.
\]

We look first at the case $c = 1$. The speeds at which the time direction $\S^1$ and the great circle in $\S^2$ are travelled are the same. Hence, the ratio at which the geodesic travels each of them is one to one. This implies that there is a unique $x\in\S^2$ such that $(x,0)\in\gamma$, for a geodesic $\gamma\in\mathcal{N}_1$.  Thus, for a null geodesic $\gamma\in\mathcal{N}_1$, let $s_0\in\R$ be such that $\gamma(s_0) = (x,0)\in\S^2\times\{0\}$. Then, $\gamma$ is completely determined by $x\in\S^2$ and by $T_{(x,0)}\pi_{\S^2}\big(\dot{\gamma}(s_0)\big)\in T_x\S^2$, where $\pi_{\S^2}:M\to\S^2$ is the projection onto the first factor. Note that $T_{(x,0)}\pi_{\S^2}\big(\dot{\gamma}(s_0)\big)\in T_x\S^2$ is simply the orthogonal projection of $\dot{\gamma}(s_0)\in T_{(x,0)}M$ onto the tangent space $T_{(x,0)}(\S^2\times\{0\})$, that is, 
\[
T_{(x,0)}\pi_{\S^2}\big(\dot{\gamma}(s_0)\big)= \dot{\mu}(s_0),
\]
which is a unit length vector by definition. We have just showed that
\[
\mathcal{N}_1 \cong ST\S^2.
\]

Let us now consider $\mathcal{N}_c$ with $c>1$. The number of turns the geodesic winds around the time direction $\S^1$ versus the number of times it winds around a great circle becomes $c$ to 1. Since all geodesics are travelled at constant speed, every $\gamma\in\mathcal{N}_c$ intersects the submanifold $\S^2\times\{0\}$ exactly at $c$ points, equidistantly spread over the great circle $\mu = \pi_{\S^2}(\gamma)$.   

Let $x\in\S^2$ and $u\in ST_x\S^2$. Then, $u$ defines a unique great circle $\mu$ in $\S^2$ such that $\mu(0) = x$ and $\dot{\mu}(0) = u$, which is parametrized by the arc. There is also a unique null geodesic $\gamma$ in $M$ such that $\pi_{\S^2}(\gamma(s) ) = \mu(s)$ with $\gamma(0) = (x,0)$. It is clear that $\dot{\gamma}(0)\in T_{(x,0)}M$ is the unique vector on the cone $C_{(x,0)}$ that orthogonally projects onto $(u,0)\in T_{(x,0)}M\cong T_x\S^2\times T_0\S^1$ with respect to $T_{(x,0)}(\S^2\times\{0\})$. Now, the geodesic $\gamma$ intersects $\S^2\times\{0\}$ exactly at $\big(\mu(\frac{2\pi j}{c}), 0\big)=:(x_j,0)\in M$ for $j = 0,\ldots, c-1$. The vector $\dot{\gamma}(\frac{2\pi j}{c})\in T_{\gamma(\frac{2\pi j}{c})}M$ is the unique vector on $C_{\gamma(\frac{2\pi j}{c})}$ with the property that $T_{\gamma(\frac{2\pi j}{c})}\pi_{\S^2}(\dot{\gamma}(\frac{2\pi j}{c})) = \dot{\mu}(\frac{2\pi j}{c}):=u_j$. That is, $\dot{\gamma}(\frac{2\pi j}{c})$ is the only vector on the light cone of $\gamma(\frac{2\pi j}{c})$ whose orthogonal projection onto $T_{\gamma(\frac{2\pi j}{c})}(\S^2\times\{0\})$ is $(u_j, 0)$.

Hence, we have just showed that any pair $(x,u)\in ST\S^2$ defines a geodesic in $\mathcal{N}_c$. Conversely, any null geodesic can be described by picking one such pair, but the choice is not unique. Indeed, any pair of the form $(x_j, u_j)$ as defined above describes the same null geodesic as $(x,u)$. Hence, we need to identify all such points in $ST\S^2$ in order to get a proper description of $\mathcal{N}_c$. This can be done as follows. 

Note that $x_j$ is obtained by a $\frac{2\pi j}{c}$ radians rotation of $x_j$ on $\mu$, which is exactly the same as a rotation of $\frac{2\pi j}{c}$ radians about the axis $x\times u\in\S^2$, regarding $(x,u)\in ST\S^2\subset\S^2\times\S^2$. It is clear that 
\begin{equation}
(x_j, u_j) = \begin{pmatrix}
	\cos\frac{2\pi j}{c} & \sin\frac{2\pi j}{c}\\
	-\sin\frac{2\pi j}{c}&
	\cos\frac{2\pi j}{c}
\end{pmatrix}\begin{pmatrix}
x\\u
\end{pmatrix}.
\end{equation}

Hence, if we define the $\Z_c$ action on $ST\S^2$ generated by 
\begin{equation}
	\label{ActionZc}
(y, v)\mapsto \begin{pmatrix}
	\cos\frac{2\pi}{c} & \sin\frac{2\pi }{c}\\
	-\sin\frac{2\pi }{c}&
	\cos\frac{2\pi }{c}
\end{pmatrix}\begin{pmatrix}
	y\\v
\end{pmatrix},
\end{equation}
we have showed that
\[
\mathcal{N}_c\cong ST\S^2/\Z_c,
\]
where $ST\S^2/\Z_c$ denotes the orbit space of $ST\S^2$ under the action of $\Z_c$. We shall now see that this space is indeed a differentiable manifold. We will make use of the following well-known result.

\begin{Definition}
	Let $N$ be a connected manifold. We say that a discrete Lie group $\Gamma$ acts on $N$ 
	\begin{enumerate}
	\item smoothly if $\sigma: N\mapsto N$ is smooth for all $\sigma\in\Gamma$,
	\item properly if, for any compact subset $K\subset N$, the set $\{\sigma\in\Gamma\ |\ (\sigma K)\cap K\neq\emptyset\}$ is compact,
	\item freely if the only element in $\Gamma$ that fixes all $N$ is the identity. 
	\end{enumerate}
\end{Definition}

\begin{Proposition}
	\cite[Thm. 21.13]{leesmooth} Let $N$ be a connected manifold. Let $\Gamma$ be a discrete Lie group acting on $N$ smoothly, properly and freely. Then, the orbit space $N/\Gamma$ is a topological manifold and has a unique smooth structure such that the projection $p:N\to N/\Gamma$ is a submersion.
\end{Proposition}

It is known that $\Z_c $ is a discrete Lie group, see \cite[Ex. 7.3 m)]{leesmooth}, and it is clear from Equation \eqref{ActionZc} that the defined action is smooth. Also, since $\Z_c$ is finite, its action on $ST\S^2$ is proper. Finally, it also follows from Equation \eqref{ActionZc} that the action is free. Thus, the quotient space $ST\S^2/\Z_c$ is a smooth manifold.

\section{A Quaternionic Approach to the Hopf Fibration and $ST\S^2$}
\label{Quaternionic}
We present now some results on the relation between the three-sphere $\S^3$ and $ST\S^2$ that will be useful for the rest of our discussion. It is well known that $\S^2$ is a double cover of $ST\S^2$. Here, we show how this covering relates to the Hopf fibration. We believe the clearest way to get to the needed results is via the use of the real division algebra of quaternions, which we denote by $\H$. In particular, we show that there exists a Hopf fibration for every unit-length pure imaginary quaternion.

Let $\V$ be the vector space of pure imaginary quaternions. Note that $\V$ can be canonically identified with $\R^3$ via the isomorphism $$ai+bj+ck\mapsto(a,b,c).$$

This identification provides $\V$ with a cross product induced by the cross product on $\R^3$, given by
\[
u\times v = \frac{uv-vu}{2}
\]
for all $u,v\in\V$. 

Let $^*:\H\to\H$ be the conjugation on $\H$. This operation allows us to define a norm on $\H$ given by $|q|^2 = qq^*$. The restriction of such norm on $\V$ induces, via the polarisation identity, an inner product on $\V$ given by
\[
\langle u,v\rangle = -\frac{uv+vu}{2}.
\]
Note that this is precisely the inner product induced by the Euclidean product in $\R^3$.

One can also identify $\S^3\cong S\H := \{q\in\H\ |\ |q| = 1\}$ and $\S^2\cong S\V : = \{u\in\V\ |\ \langle u,u\rangle = -u^2 = 1\}$. Finally, one has $ST\S^2\cong ST(S\V):=\{(u,v)\in S\V\times S\V\ |\ \langle u,v\rangle = 0\}$.

\begin{Proposition}
	\label{Difeoprimer}
	Let $u,v\in S\V$ be such that $\langle u,v\rangle= 0$. Then, there exists a surjective local diffeomorphism
	\[
	\begin{array}{cccc}
		\Phi_{(u,v)}:&S\H&\to&ST(S\V)\\
		& q&\mapsto&(quq^{-1},qvq^{-1})
	\end{array}
	\]
	such that the preimage of a point in $ST(S\V)$ consists of exactly two antipodal points of $S\H$.
\end{Proposition}
\begin{proof}
	Let us first show that $\Phi_{(u,v)}$ is well defined. We can compute 
	\[
	\langle quq^{-1}, quq^{-1}\rangle = -(quq^{-1})(quq^{-1}) = -qu^2q^{-1} = qq^{-1} = 1,
	\]
	and similarly for $\langle qvq^{-1}, qvq^{-1}\rangle  =1$. In addition, twice the real part of $quq^{-1}$ is 
	\[
	quq^{-1}+(quq^{-1})^* = quq^{-1}-(q^{-1})^*uq^* = quq^{-1}-quq^{-1} = 0,
	\]
	and similarly for $qvq^{-1}$.
	
	Also, 
	\[\langle quq^{-1},qvq^{-1}\rangle = -\frac{quq^{-1}qvq^{-1}+qvq^{-1}quq^{-1}}{2} = -q\frac{uv+vu}{2}q^{-1} = 0,
	\]
	and thus the map is well defined.
	
	Let us now show surjectivity. Let $(w,z)\in ST(S\V)$. Assume that $u$ and $w$ are not colinear and let $\theta = \arccos\langle u, w\rangle\in (0,\pi)$. Let $q_1 = e^{\frac{\theta}{2}\frac{u\times w}{|u\times w|}} = \cos\frac{\theta}{2}+\frac{u\times w}{|\sin\theta|}\sin\frac{\theta}{2}$. Assume now that $q_1vq_1^{-1}$ and $z$ are not colinear and let $\tau = \arccos\langle q_1vq_1^{-1},z\rangle\in (0,\pi)$. Let $q_2 = e^{\frac{\tau}{2}\frac{q_1vq_1^{-1}\times z}{|q_1vq_1^{-1}\times z|}} = \cos\frac{\tau}{2}+\frac{q_1vq_1^{-1}\times z}{|\sin\tau|}\sin\frac{\tau}{2}$. Let $q = q_2q_1\in S\H$. We will now make use of the following technical lemma. 
	
	\begin{Lemma}
		Let $z\in S\V$ and $s = \cos\frac{\alpha}{2}+\sin\frac{\alpha}{2}z$. Then, for all $r\in S\V$, 
		\[
		srs^{-1} = r\cos\alpha+(z\times r)\sin\alpha+z\langle z, r\rangle(1-\cos\alpha).
		\]
	\end{Lemma}
	
	Hence, we can now compute
	\begin{align*}
		quq^{-1} = q_2(q_1uq_1^{-1})q_2^{-1} &= q_2\Big(u\cos\theta+\Big((u\times w)\times u\Big)\frac{\sin\theta}{|\sin\theta|}+(u\times w)\langle u\times w,u\rangle\frac{1-\cos\theta}{\sin^2\theta}\Big)q_2^{-1}  \\ &=q_2\Big(u\langle u,w\rangle + w-u\langle u,w\rangle\Big)q_2^{-1} = q_2wq_2^{-1} \\ & = w\cos\tau+\Big((q_1vq_1^{-1})\times z\Big)\times w+(q_1vq_1^{-1}\times z)\langle q_1vq_1^{-1}\times z, w\rangle \frac{1-\cos\tau}{\sin^2\tau} \\ & =w\cos\tau+(q_1vq_1^{-1}\times z)\langle q_1vq_1^{-1}\times z, w\rangle \frac{1-\cos\tau}{\sin^2\tau} = w\cos\tau+w(1-\cos\tau) = w,
	\end{align*}
	and
	\begin{align*}
		qvq^{-1} = q_2(q_1vq_1^{-1})q_2^{-1} & = q_1vq_1^{-1}\cos\tau +(q_1vq_1^{-1}\times z)\times q_1vq_1^{-1}+\frac{q_1vq_1^{-1}\times z}{\sin\tau^2}\langle q_1vq_1^{-1}\times z,q_1vq_1^{-1}\rangle(1-\cos\tau)  \\& = q_1vq_1^{-1}\cos\tau +(q_1vq_1^{-1}\times z)\times q_1vq_1^{-1} = q_1vq_1^{-1}\cos\tau+z-q_1vq_1^{-1}\cos\tau = z.
	\end{align*}
	
	If, after conjugating by $q_1$, one finds that $q_1vq_1^{-1}$ and $z$ are colinear, then $q_1vq_1^{-1} = z$ or $q_1vq_1^{-1} = -z$. For the first case, take $q_2 = 1$. For the second case, take $q_2 = e^{\frac{\pi}{2}w}$. If, from the beginning, $w$ and $u$ are colinear, take $q_1 = 1$ if $w = u$ and $q_1 = e^{\frac{\pi}{2}z}$ if $w = -u$. Then, find $q_2$ as showed. Thus, the mapping is onto.

	Let us now compute $\Phi_{(u,v)}(-q) = \Big((-q)u(-q)^{-1}, (-q)v(-q)^{-1}\Big) = (quq^{-1},qvq^{-1}) = \Phi(q)$.
	
	Let now $q,p\in S\H$. Let then $q = e^{\theta w}$ and $p = e^{\tau z}$, with $\theta,\tau \in[0,2\pi)$ and $w,z\in\V$. Then, using the previous technical claim, it is not hard to see that $\Phi_{(u,v)}(q) = \Phi_{(u,v)}(p)$ implies that either $\theta w= \tau z$ or $\theta w = (\pi+\tau)z$, which concludes the proof.
	
	Differentiability of the mapping and its local inverse (constructed as showed in this proof) follow from taking charts in $S\H\subseteq\H\cong\R^4$ and $ST(S\V)\subseteq S\V\times S\V\subseteq \V\times\V\cong\R^3\times\R^3$.
\end{proof}

It is clear now how any pair $(u,v)\in ST(S\V)$ provides a diffeomorphism from $S\H/\Z_2$ onto $ST(S\V)$.
\begin{Corollary}
	Let $\Z_2$ act on $S\H$ via the antipodal map. Then, for all $u,v\in S\V$ such that $\langle u,v\rangle = 0$, there exists a diffeomorphism $S\H/\Z_2\cong ST(S\V)$.
	Hence, for all $u,v\in \S^3$ such that $\langle u,v\rangle = 0$. there exists a diffeomorphism $\S^3/\Z_2\cong ST\S^2$. We will denote both of these maps by $\Phi_{(u,v)}$ when the context makes clear which of them is being used.
\end{Corollary}
\begin{proof}
	The result follows from the fact that $\Phi_{(u,v)}$ is a local diffeomorphism and that the preimage of a point in $S\V$ is precisely a pair of antipodal points. The differentiability of such map and its inverse follows from the differentiability of $\Phi_{(u,v)}$ and its local inverse.
\end{proof}

Note that choosing orthogonal vectors $u, v\in ST(S\V)$ is simply choosing the image of $1\in S\H$. Once this image is fixed, the diffeomorphism is completely determined. Equivalently, in terms of $\S^3\cong S\H $ and $ ST\S^2\cong ST(S\V)$, it is enough to pick $v\in T_u\S^2\subset T\S^2$ as the image of the north pole of $\S^3$ (or any other point of the three-sphere as a matter of fact) to completely determine the diffeomorphism between $\S^3/\Z_2\cong \R P^3$ and $ST\S^2$.

As seen at the end of Section \ref{UTB}, $ST\S^2$ is an $\S^1$-bundle over $\S^2$ via the obvious projection 
\[
\begin{array}{cccc}
	\pi: & ST\S^2&\to & \S^2\\
	&(x,u)&\mapsto& x.
\end{array}
\]

However, there is a more interesting projection
\[
\begin{array}{cccc}
	f:&ST\S^2&\to &\S^2\\
	&(x,u)&\mapsto&x\times u.
\end{array}
\]

Hence, $ST(S\V)$ can be given a structure of $\S^1$-bundle over $S\V$ via
\[
\begin{array}{cccc}
	f:&ST(S\V)&\to &S\V\\
	&(u,v)&\mapsto&u\times v.
\end{array}
\]

Let us now show how these fibres look like.

\begin{Lemma}
	Let $q\in S\V$. Then, the fibre of $q$ under $f:ST(S\V)\to S\V$ is $$\{(e^{-q\theta}ue^{q\theta}, e^{-q\theta}ve^{q\theta})\in ST(S\V)\ |\ \theta\in \R\}$$ for any $(u,v)\in f^{-1}(q)$.
\end{Lemma}
\begin{proof}
	Let us compute
	\begin{align*}
		f(e^{-q\theta}ue^{q\theta}, e^{-q\theta}ve^{q\theta}) &= (e^{-q\theta}ue^{q\theta})\times(e^{-q\theta}ve^{q\theta}) \\ & =  \frac{1}{2}\Big[e^{-q\theta}uve^{q\theta}-e^{-q\theta}vue^{q\theta}\Big]  \\&= e^{-q\theta}\frac{uv-vu}{2}e^{q\theta}  \\& = e^{-q\theta}(u\times v)e^{q\theta}\\ & = e^{-q\theta}qe^{q\theta} = q = u\times v.
	\end{align*}
\end{proof}

This lemma provides any fibre of $f: ST(S\V)\to S\V$ with a closed operation $\pi^{-1}(p)\times\pi^{-1}(p)\to\pi^{-1}(p)$ for all $p\in S\V$ via the identification $\pi^{-1}(p)\cong\S^1$.  

The next result shows that, for all $w\in S\V$, we can obtain a Hopf fibration of  $S\H\cong\S^3$ over $S\V\cong\S^2$.

\begin{Proposition}
	Let $w\in S\V$. Then, the projection
	\[
	\begin{array}{cccc}
		\tau_w:&S\H&\to&S\V\\
		& q&\mapsto&qwq^{-1}
	\end{array}
	\]
	provides $S\H$ with a structure of $\S^1$-bundle over $S\V$. The fibre of $p\in S\V$ is $\{qe^{w\theta}\ |\ \theta\in\R\}$, for $q\in\tau_w^{-1}(p)$.
\end{Proposition}
\begin{proof}
	The proof of Proposition \ref{Difeoprimer} shows that $\tau_w$ is well defined and surjective.
	
	Let $q\in S\H$ and $\theta\in\R$. Then, \[
	\tau_w(qe^{w\theta}) = qe^{w\theta}w(qe^{w\theta})^{-1} = qe^{w\theta}w(e^{w\theta})^{-1}q^{-1} = qwe^{w\theta}(e^{w\theta})^{-1}q^{-1} = qwq^{-1} = \tau_w(q)
	\]
\end{proof}

Similarly to the fibres of $\pi: ST(S\V)\to S\V$, this last result provides any fibre $\tau_w^{-1}(p)$ with a closed operation $\tau_w^{-1}(p)\times \tau_w^{-1}(p)\to\tau_w^{-1}(p)$ via the identification $\tau_w^{-1}(p)\cong \S^1$ for all $p\in S\V$.

\begin{Lemma}
	The $\Z_2$-action on $S\H$ given by the antipodal map preserves the fibres of $\tau_w: S\H\to S\V$. Hence, for all $w\in S\V$, the map
	\[
	\begin{array}{cccc}
		\tilde{\tau}_w:&S\H/\Z_2&\to&S\V\\
		&[q]&\mapsto&\tau_w(q)
	\end{array}
	\]
	provides $S\H/\Z_2$ with a structure of $\S^1/\Z_2\cong\S^1$-bundle over $S\V$. Making an abuse of notation, we will continue denoting $\tilde{\tau}_w = \tau_w$ when there is no possible confusion.
\end{Lemma}
\begin{proof}
	Trivially, if $q\in S\H$, one has $\tau(-q) = (-q)w(-q)^{- 1} = qwq^{-1} = \tau(q)$.
\end{proof}

We conclude this section with the following result,

\begin{Proposition}
	\label{propcomm}
	Let $u,v\in S\V$ such that $\langle u, v\rangle = 0$ and let $\Phi_{(u,v)}:S\H/\Z_2\to ST(S\V)$ be their induced diffeomorphism. Let $\tau_w:S\H/\Z_2\to S\V$ be the projection induced by $w = u\times v\in S\V$. Then, the following diagram commutes.
	\[
	\begin{tikzcd}
		S\H/\Z_2\arrow{rr}{\Phi_{(u,v)}}\arrow{dr}{\tau_w} && ST(S\V)\arrow{dl}[above]{f}\\
		&S\V&
	\end{tikzcd}
	\]
	In addition, for all $p\in S\V$, $\Phi_{(u,v)}$ induces a diffeomorphism between $\tau_w^{-1}(p)$ and $f^{-1}(p)$.
\end{Proposition}
\begin{proof}
	Let $q\in S\H$. Then, 
	\begin{align*}
		f(\Phi_{(u,v)}(q)) =& f(quq^{-1}, qvq^{-1}) = (quq^{-1})\times(qvq^{-1}) = \frac{quq^{-1}qvq^{-1}-qvq^{-1}quq^{-1}}{2}  \\ =& \frac{quvq^{-1}-qvuq^{-1}}{2} = q\frac{uv-vu}{2}q^{-1} = q(u\times v)q^{-1} = \tau_w(q).
	\end{align*}
	
	Since $\Phi_{(u,v)}$ is bijective, for all $p\in S\V$, it induces a bijection $\Phi_{(u,v)}|_{\tau_w^{-1}(p)}:\tau^{-1}(p)\to f^{-1}(p)$. Differentiability of this map and its inverse follow from the differentiability of $\Phi_{(u,v)}$ and $\Phi_{(u,v)}^{-1}$.
\end{proof}
\section{The Spaces of Null Geodesics as Lens Spaces}
\label{Lens}

We will now make use of the results discussed in Section \ref{Quaternionic} to show that the spaces of null geodesics of the spacetimes $(\S^2\times \S^1, g_c)$ are lens spaces.

\begin{Definition}
	Consider $\S^3\subseteq\mathbb{C}\times\mathbb{C}$. Let $p,q\in\mathbb{Z}^+$ be coprime integers. Consider the $\mathbb{Z}_p$ action on $\S^3$ generated by
	\[
	(z_0,z_1)\mapsto (e^{\frac{2\pi i}{p}}z_0, e^{\frac{2\pi i q}{p}}z_1).
	\]
	The lens space $(p,q)$ is $L(p,q) = \S^3/\mathbb{Z}_p$, with the induced differentiable and topological structures.
\end{Definition}

Consider the lens space $L(p,p-1)$. The $\Z_p$ action that generates it is 
\[
(z_0,z_1)\mapsto (e^{\frac{2\pi i}{p}}z_0, e^{-\frac{2\pi i}{p}}z_1).
\]

One can also construct this space as a quotients of $S\H\cong\S^3$. Indeed, using the canonical isomorphism
\[
(\mu+ ai, b+ci)\mapsto \mu+ai+bj+ck
\]
between $\R^4 \cong \C\times\C$ and $\H$, the three-sphere $\S^3$ gets canonically identified with $S\H$, as stated at the beginning of Section \ref{Quaternionic}. Note that by letting $z_0 = \mu+ai, z_1 = b+ci\in\C$, this isomorphism can be rewritten as
\[
(z_0,z_1)\mapsto z_0+z_1j.
\]

Hence, one can translate the $\Z_p$ action on $\S^3$ that defines the lens space $L(p,p-1)$ onto the unit quaternions $S\H$ as 
\[
z_0+z_1j\mapsto e^{\frac{2\pi i}{p}}z_0+e^{-\frac{2\pi i}{p}}z_1j =z_0e^{\frac{2\pi i}{p}} +z_1je^{\frac{2\pi i}{p}} =  (z_0+z_1j)e^{\frac{2\pi i}{p}},
\]
for $z_0,z_1\in\C$, that is, 
\[
q\mapsto qe^\frac{2\pi i}{p}.
\]
Then, trivially, $L(p,p-1)\cong S\H/\Z_p$.

Note that the $\Z_2$-action on $S\H$ that defines the lens space $L(2,1)$ is precisely the given by the antipodal map. Hence, we have showed

\begin{Proposition}
	We have $ST(S\V)\cong S\H/\Z_2\cong L(2,1)$. Also, $ST\S^2 \cong L(2,1) = \R P^3$.
\end{Proposition} 

The theory developed in Section \ref{Quaternionic} allows us to formalise and prove the following key result in our discussion.

\begin{Proposition}
	Let $\Phi_{(j, k)}$ be the diffeomorphism between $S\H/\Z_2$ and $ST(S\V)$ induced by the pair $(j, k)\in S \V\times S\V$. Then, the $\Z_{2c}$-action on $S\H$ that generates the lens space $L(2c,2c-1)$ descends to a $\Z_c$-action on $S\H/\Z_2$ that, via $\Phi_{(j,k)}$, induces the $\Z_c$-action on $ST(S\V)$ generated by 
	\[(u,v) \mapsto \begin{pmatrix}
		\cos\frac{2\pi}{c} &  	\sin\frac{2\pi}{c} \vspace{.1cm}\\
		-\sin\frac{2\pi}{c}&  	\cos\frac{2\pi }{c}
	\end{pmatrix}\begin{pmatrix}
		u\\v
	\end{pmatrix}.\]
	Hence, $\Phi_{(j,k)}$ induces a diffeomorphism between $S T(S\V)/\mathbb{Z}_c$ and $L(2c,2c-1)$.
\end{Proposition}
\begin{proof}
	Note first that, since the action that defines the lens space preserves the fibres of the map $\tau_{i}:S\H\to S\V$ induced by $i= j\times k$, Proposition \ref{propcomm} assures that the induced action on $ST(S\V)$ will preserve the fibres of $f$, just like the action defined on $ST(S\V)$ in the statement of the corollary.
	
	Let $q\in S\H$. The $\Z_{2c}$-action on $S\H$ that defines the lens space is generated by $q\mapsto qe^{\frac{\pi i}{c}}$. Then, $\Phi_{(j,k)}(qe^{\frac{\pi i}{c}}) = \Big(qe^{\frac{\pi i}{c}}je^{-\frac{\pi i}{c}}q^{-1}, qe^{\frac{\pi i}{c}} k e^{-\frac{\pi i}{c}}q^{-1}\Big) = \Big(qje^{\frac{-2 \pi i}{c}}q^{-1}, qke^{\frac{-2 \pi i}{c}}q^{-1}\Big)$.
	
	Now, consider $\Phi_{(j,k)}(q) = (qjq^{-1}, qkq^{-1})$. Then, 
	\begin{align*}
		\cos\frac{2\pi}{c}qjq^{-1}+\sin\frac{2\pi}{c}qkq^{-1} = qj\Big(\cos\frac{2\pi}{c}-i\sin\frac{2\pi}{c}\Big)q^{-1} = qje^{-\frac{2\pi i }{c}}q^{-1},
	\end{align*}
	and
	\begin{align*}
		-\sin\frac{2\pi}{c}qjq^{-1}+\cos\frac{2\pi}{c}qkq^{-1} = qk\Big(\cos\frac{2\pi}{c} - i\sin\frac{2\pi}{c}\Big)q^{-1} = qke^{-\frac{2\pi i}{c}}q^{-1},
	\end{align*}
	as needed.
\end{proof}

Let us present the following result on lens spaces.
\begin{Proposition} \cite{bonahon}
	Let $p\in\mathbb{Z}^+$ and $q,q'\in\Z^+$ coprime with $p$. Suppose that $q = \pm q'\mod p$. Then, $L(p,q)\cong L(p,q')$. 
\end{Proposition}

Thus, since $-(2c-1) = -2c+1 = 1\mod 2c$, we find 
$$L(2c,2c-1) \cong L(2c,1).$$ Hence, we have showed the most important result of this section,
\begin{Theorem}
	\label{Bigone}
	Let $(\S^2,g_\circ)$ be the unit sphere in $\R^3$ with its induced metric. Consider the manifold $M = \S^2\times\S^1$ and, being $t$ the coordinate on $\S^1$, define the Lorentzian metric on $M$ given by $g_c = g_\circ-\frac{1}{c^2}dt^2$. Let $\mathcal{N}_c$ be the space of null geodesics on $M$ under the metric $g_c$. Then,
	\[
	\mathcal{N}_c\cong L(2c,1).
	\]
\end{Theorem}

We conclude this section with the following remark. Note that we have showed that $L(4,1)\cong S T\S^2/\Z_2$, where the action of $\Z_2$ on $S T\S^2$ is the one generated by $(\sigma,d\sigma)$, where $\sigma:\S^2\to\S^2$ is the antipodal map. If we denote by $r:ST\S^2\to ST\S^2/\Z_c$ the projection, such an action induces a $\Z_2$ action on $\S^2$ given by 
\[
x\mapsto r\big((\sigma,d\sigma)(x,v)\big) = \sigma(x) = -x,
\]
for any $v\in S T_x\S^2$, i.e. the antipodal map. Thus, $L(4,1)$ is also an $\S^1/\Z^2\cong \S^1$ fibration over $\S^2/\Z_2 = \R P^2$. Since the fibration is induced by the canonical fibration of $S T\S^2$ on $\S^2$, we have showed that $L(4,1)\cong ST\R P^2$, the unit tangent bundle of the projective plane (for more details, see \cite{unitprojective}).

\section{The Contact Structure on $\mathcal{N}_c$}
\label{contactsttuctures}
In this section we explicitly compute the contact structure on the spaces $\mathcal{N}_c$. We will show how the contact structure on $\mathcal{N}_1\cong ST\S^2$ is precisely the canonical contact structure of $ST\S^2$ defined in Section \ref{UTB}, and that the contact structure on $\mathcal{N}_c$ for $c>1$ is the pushforward of this under the projection mapping $r:ST\S^2\to L(2c,1)$.

Let us first consider the case $c = 1$. Let $\gamma\in\mathcal{N}_1\cong ST\S^2$. Recall that $\gamma$ is given by the lift-up of the great circle $\mu:\R\to \S^2$ defined by the pair $(x,u)\in ST\S^2$ representing $\gamma$, that is, such that $\mu(0) = x$ and $\dot{\mu}(0) = u$. We will show that 
\[
\mathcal{H}_\gamma = \chi_{(x,u)},
\]
where $\chi$ is the canonical contact structure of $ST\S^2$. Take $(x,0)\in\gamma$ and $\gamma(\tau)\neq (x,0)$, with $0<\tau<\pi$.

Note that all geodesics in the sky $\mathfrak{S}_{(x,0)}$ intersect $\S^2\times\{0\}$ at $(x,0)$, so $\pi(\mathfrak{S}_{(x,0)}) = \{x\}$. Hence, $T_{x}\pi(T_\gamma\mathfrak{S}_{(x,0)}) = \lspan 0\rspan \subseteq \lspan u\rspan^\perp$. Thus, $T_\gamma\mathfrak{S}_{(x,0)}\subset \chi_{(x,u)}$.

Consider now $\gamma(\tau)\neq (x,0)$. Since $c = 1$, we know that $\gamma(s) = \big(\mu(s), s\big)$ and hence $y: = \pi_{\S^2}\big(\gamma(\tau)\big) = \mu(\tau)$. Let $v\in ST_y\S^2$ such that $\langle\dot{\mu}(\tau), v\rangle = 0$. Since all geodesics in $\S^2\times\S^1$ are travelled at the same speed, it is clear that the projection of the sky of $\gamma(\tau)$ is parametrized by
\[
\pi(\mathfrak{S}_{\gamma(\tau)})(s) = y\cos\tau+(\dot{\mu}(\tau)\cos s+v\sin s)\sin\tau,
\]
which is the set of points in $\S^2$ at distance $\tau$ of $y$, and $\pi(\gamma) = \pi(\mathfrak{S}_{\gamma(\tau)})(0)$. Hence, 
\[
T_x\pi(T_\gamma\mathfrak{S}_{\gamma(\tau)}) = \lspan\frac{d}{ds}\big|_{s = 0}\big(y\cos\tau+(\dot{\mu}(\tau)\cos s+v\sin s)\sin\tau\big)\rspan = \lspan v\sin\tau\rspan. 
\]

Since $\mu$ is the great circle defined by $x$ and $u$, we have $\mu(\R) = \lspan x, u\rspan \cap\S^2$. Since $v$ is orthogonal to $\dot{\mu}(\tau)$, it is orthogonal to $ \lspan x, u \rspan$, and hence $T_x\pi(T_\gamma\mathfrak{S}_{\gamma(\tau)})  = \lspan v\sin\tau\rspan\subset\lspan u\rspan^\perp$. This implies that $T_\gamma\mathfrak{S}_{\gamma(\tau)}\subset \chi_{(x,u)}$.

\begin{figure}[H]
	\centering
	\includegraphics[width = \textwidth, trim=4 200 4 4,clip]{geodesiques2}
\end{figure}

Since the canonical contact structure of $ST\S^2$ and the contact structure on $ST\S^2$ coming from it being a space of null geodesics both have rank 2, we have showed the following.

\begin{Theorem}
	The contact structure on $\mathcal{N}_1 = ST\S^2$ is the canonical contact structure on $ST\S^2$ defined in Section \ref{UTB}.
\end{Theorem}

Consider now the case $c>1$. Then, $\mathcal{N}_c \cong L(2c,1)$. As discussed in Section \ref{Lens}, the lens spaces $L(2c,1)\cong \mathcal{N}_c$ can be obtained as a quotient of $ST\S^2$ by the action of $\Z_c$. Let
\[
\begin{array}{cccc}
	r:& ST\S^2\to L(2c,1) \cong \mathcal{N}_c
\end{array}
\]
be the projection map, which provides a surjective local diffeomorphism. In order to simplify the notation, we will denote an element of $ST\S^2$ by $u\in ST\S^2$, understanding that $u\in ST_{\pi(u)}\S^2$. Let us also denote by $[u]\in L(2c, 1)$ the class of $u$ under the action of $\Z_c$.

Let $[u]\in L(2c,1)$ and let $U'$ be an open subset of $\S^2$ such that $u$ is the only preimage of $[u]$ in $U:=\pi^{-1}(U')$. Then, $U$ can be taken to be small enough so that
\[
r|_{U}:U\to r(U),
\]
is a diffeomorphism. We will show that $$\mathcal{H}_{[u]} = (r|_U)_*\chi_{u},$$
where $r_*$ is the pushforward of $r$.

It is necessary to show first that the pushforward of $\chi$ is well defined, that is, 
$$(r|_U)_*\chi_u= (r|_V)_*\chi_v$$ for $u,v\in r^{-1}([u])$, and taking open subsets $U',V'\subset\S^2$ as described above in order to restrict $r$ to a diffeomorphism on $U = \pi^{-1}(U')$ and $V = \pi^{-1}(V')$.

Consider $\chi_u$. It is enough to show that two linearly independent vectors in $\chi_{u}$ map to linearly independent vectors in $\chi_{v}$ via
\[
(r|^{-1}_V)_*\circ (r|_U)_*.
\]

Note that the mapping $r|_V^{-1}\circ r|_U$ need not be defined for all points in $U$, as there might be points in $r(U)$ not in $r(V)$. However, it is well defined for a neighbourhood $W\subset ST\S^2$ of $u$, which is all that is needed.

Assume $x = \pi(u)$ and $y = \pi(v)$ are not antipodal points, and take a curve $$\alpha:(-\varepsilon, \varepsilon)\to W\subseteq ST\S^2$$ such that $\alpha(0) = (x,u)$ and the image of $\alpha$ lies entirely in the fibre of $x\in\S^2$. Then, $\dot{\alpha}(0)\in\chi_{u}$.  Let $\beta = r|_V^{-1}\circ r|_U\circ\alpha$. It is clear that $\beta(0) = v$. We will now show that $\dot{\beta}(0)\in \chi_{v}$. Consider the curve $\pi\circ\beta$ on $\S^2$. Let $u'\in ST_x\S^2$ such that $\langle u,u'\rangle = 0$. Then, $\pi\circ\beta$ is an arc of the circle parametrized by
\[
\phi(s) = x\cos\tau +\sin\tau(u \cos s + u'\sin s),
\]
where $\tau\in(0,\pi)$ is such that $\phi(0) = x\cos\tau +\sin\tau u = y$. We can compute 
\[
T_{v}\pi\big(\dot{\beta}(0)\big) = \frac{d}{ds}\big|_{s = 0}\phi(s) =  u'\sin\tau \in\lspan v\rspan^\perp,
\]
where the fact that $ u'\sin\tau\in\lspan v\rspan^\perp$ follows from the same argument as above. Thus, $\dot{\beta}(0)\in\chi_{v}$. It is also clear that $\dot{\beta}(0)$ is nonzero.

\begin{figure}[H]
	\centering
	\includegraphics[width = .9\textwidth, trim=4 4 4 4,clip]{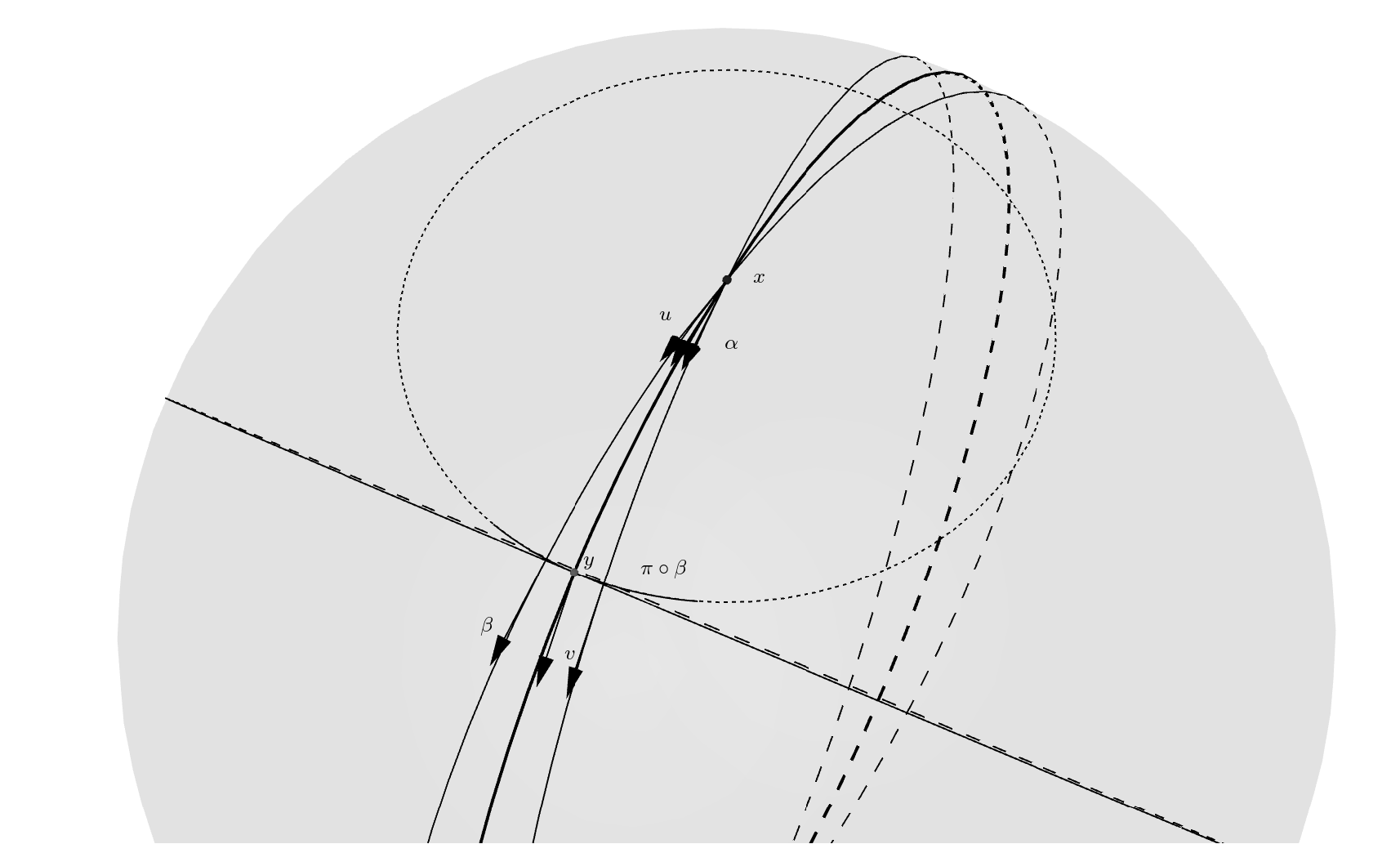}
\end{figure}

Let us now consider the following curve in $W$. Let $\mu^\perp$ be the great circle in $\S^2$ defined by $(x,u^\perp)$, where $u^\perp$ in any vector in $ST_{x}\S^2$ orthogonal to $u$, and with $\mu^{\perp}(0) = x$. Let $c:(-\varepsilon,\varepsilon)\to W$ be the curve in $W$ defined by $c(0) = (x,u)$ and $c(s) = (\mu^\perp(s),z(s))$, where $z(s)\in ST_{\mu^\perp(s)}\S^2$ is the vector tangent to the geodesic great circle that goes through the points $\mu^\perp(s)$ and $y$, pointing to the same hemisphere as $u$. It is clear that $T_{u}\pi\big(\dot{c}(0)\big)=  u^\perp\in \chi_{u}$ and that it is linearly independent to the previously defined $\dot{\alpha}(0)$. 

Let now $\theta(s) = r|_V^{-1}\circ r|_U\circ c(s)$. Clearly, $\theta(0) = (y,v)$. We will show that $\dot{\theta}(0)\in\chi_{v}$. By construction, we have $\pi\circ\theta(s) = y$, which implies that
\[
T_{v}\pi(\dot{\theta}(0))  = 0\in\lspan v\rspan^\perp.
\]

Hence, $\dot{\theta}(0)\in\chi_{v}$. Also, $\dot{\theta}(0)$ is clearly non-zero.

\begin{figure}[H]
	\centering
	\includegraphics[width = .8\textwidth]{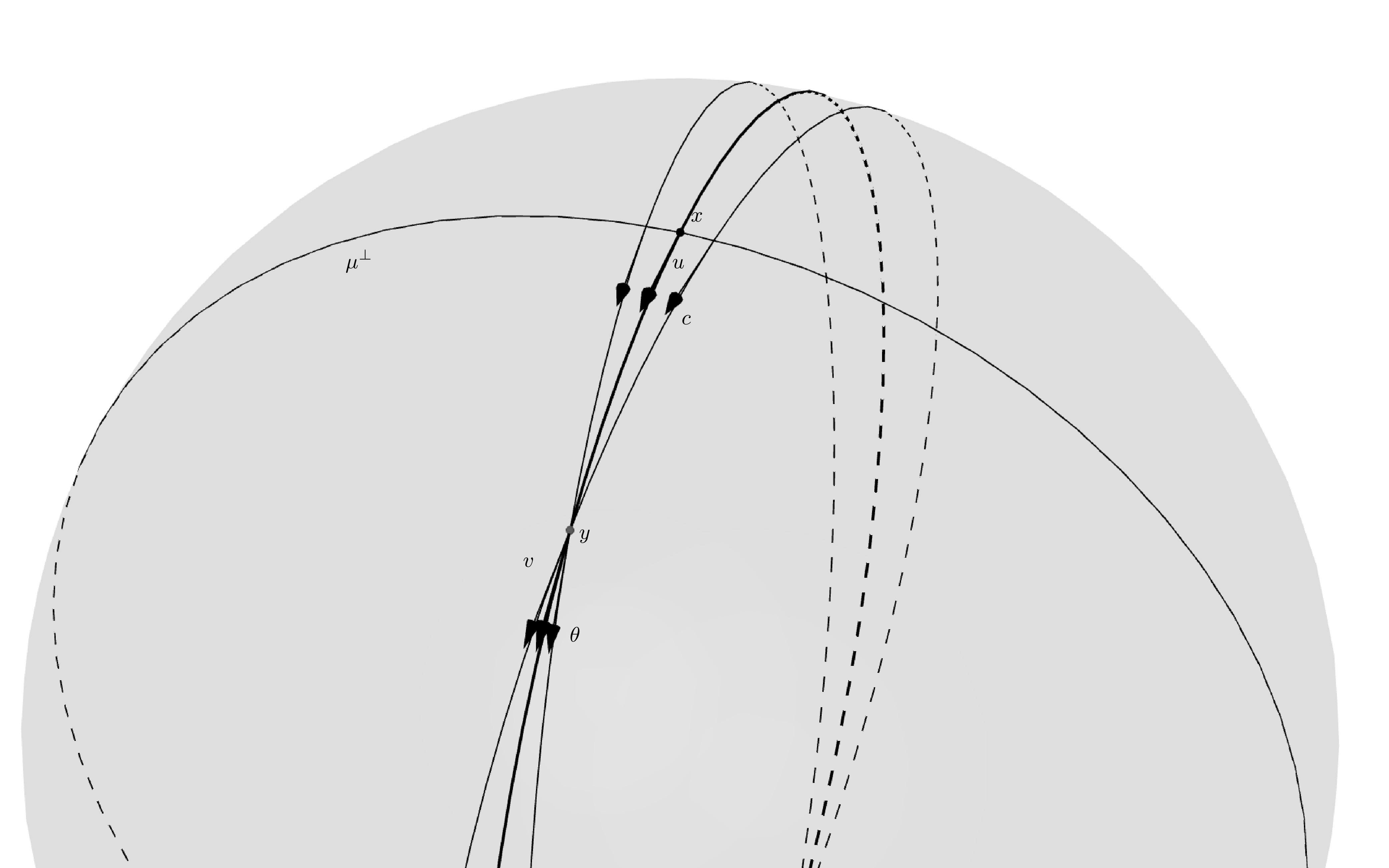}
\end{figure}

Consider now the case in which $x$ and $y$ are antipodal. Then, the image of $\alpha$ lies entirely in the fibre of $y$ and hence the projection of its tangent vector at time 0 onto $\S^2$ is null. Thus, the tangent vector lies in $\chi_{v}$. Now, take as $c$ the curve $c(0) = (\mu^\perp(s), z(s))$, where $z(s)\in ST_{\mu^\perp(s)}\S^2$ is the perpendicular vector to $\dot{(\mu^\perp)}(s)$ that points to the same hemisphere as $u$. Then, the image $\theta$ of $c$ projects onto $\S^2$ via $\pi$ to the great circle that contains $x$ and $y$ and is perpendicular to $u$ and $v$. The claim follows.

\begin{figure}[H]
	\centering
	\includegraphics[width = .9\textwidth, trim=4 20 4 4,clip]{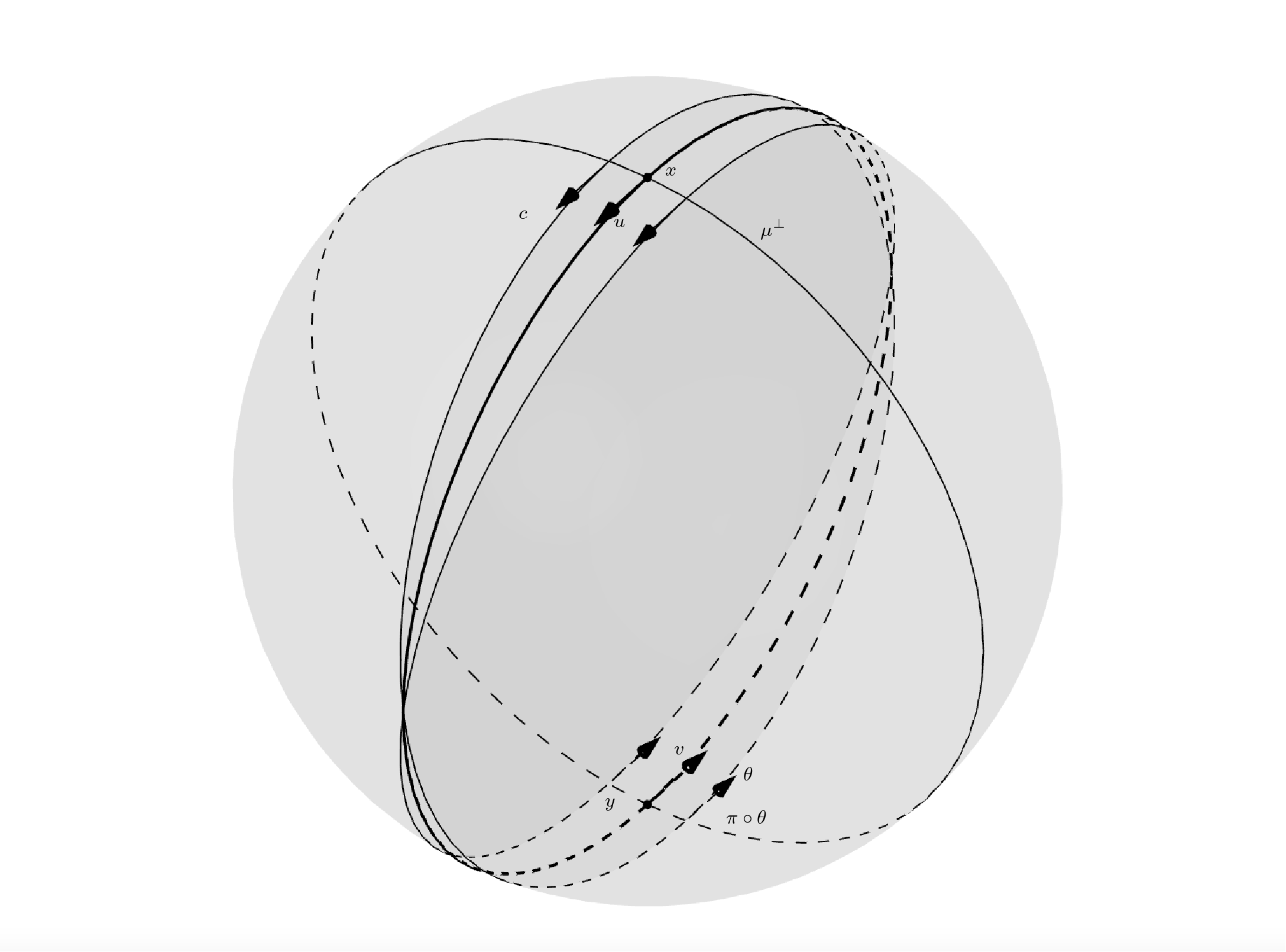}
\end{figure}

Hence, we have showed
\[
(r|_U)_*\chi_u= (r|_V)_*\chi_v,
\]
and thus the pushforward of $\chi$ is well defined.

Let us now show that $$\mathcal{H}_{[u]} = (r|_U)_*\chi_u.$$

Let $x = \pi(u)$. We know that $[u]$ describes the geodesic $\gamma$ in $\S^2\times\S^1$ that intersects $U'\times\{0\}$ only at $(x,0)$. Take $(x,0)\in\gamma$ and consider its sky $\mathfrak{S}_{(x,0)}$. It is clear that $\mathfrak{S}_{(x,0)} = \{[v]\ |\ v\in ST_{x}\S^2\}$, and thus
\begin{align*}
	\big(\pi\circ r|_{U}^{-1}\big)_* \big(T_{[u]}\mathfrak{S}_{(x,0)}\big) =& T_{x}(\{\pi(v)\ |\ v\in ST_{x}\S^2\}) \\ = & T_{x}(\{x\}) = \lspan0\rspan\subset\lspan u\rspan^\perp,
\end{align*}
from which we deduce that $$(r|_{U}^{-1})_* \big(T_{[u]}\mathfrak{S}_{(x,0)}\big)\subset \chi_{u}$$ and, thus, $$T_{[u]}\mathfrak{S}_{(x,0)}\subset (r|_U)_*(\chi_{u}).$$

Take now $\gamma(\tau)\neq(x,0)$ but close enough so that $\pi(\mathfrak{S}_{\gamma(\tau)})\subset U'$. As discussed previously, $\pi\circ  r|_U^{-1}(\mathfrak{S}_{\gamma(\tau)})$ describes a circle $\phi(s)$ in $U'$ whose tangent vector at $(x,0)$ is orthogonal to $u$. Hence,
\begin{align*}
	(\pi\circ r|_U^{-1})_* \big(T_{[u]}\mathfrak{S}_{\gamma(\tau)}\big) =& T_{x}\Big(\pi\circ r|_U^{-1}(\mathfrak{S}_{\gamma(\tau)} )\Big) \\=& T_{x}\{\phi(s)\ |\ s\in \mathbb{R}\}\subset\lspan u\rspan^\perp,
\end{align*}
which implies that $$(r|_U^{-1})_* \big(T_{[u]}\mathfrak{S}_{\gamma(\tau)}\big)\subset \chi_{u},$$ and hence $$T_{[u]}\mathfrak{S}_{\gamma(\tau)}\subset (r|_U)_*(\chi_{u}).$$

Thus, we have showed
\begin{Theorem}
	\label{contactstructure}
	Let $c\in\mathbb{N}^+$. Let $[u]\in\mathcal{N}_c\cong L(2c,1)$. Let $u\in r^{-1}([u])$ and $U'$ an open neighbourhood of $\pi(u)$ in $\S^2$ such that $u$ is the only preimage of $[u]$ in $U = \pi^{-1}(U')$ and $r|_U$ is a diffeomorphism onto its image. Then, 
	\[
	\mathcal{H}_{[u]} = (r|_U)_*(\chi_{u})
	\]
	where $\chi$ is the canonical contact structure on $ST\S^2$ defined in Section \ref{UTB}.
\end{Theorem}
\chapter{Prolongations and Deprolongations}
\label{ProlandDeprolChap}
\section{The Space of Null Geodesics as a Deprolongation}
\label{Deprol}
In Section \ref{EngelSection}, we discussed how a Lorentzian three-manifold defines a natural Engel manifold via the Lorentz prolongation. Similarly, a contact three-manifold produces an Engel manifold by means of the Cartan prolongation. The current section studies how these structures are related to each other and to the space of null geodesics of a spacetime.

 Let us present the following result, which was first derived in \cite[Thm. 4.2]{deturck} for Riemannian manifolds and later generalised to the pseudo-Riemannian case in \cite{kowalski}. 

\begin{Theorem}[\cite{kowalski}]
	Let $(M, g)$ be a pseudo-Riemannian three-manifold. Then, in a neighbourhood of every point $x\in M$, there is a local chart of $M$ in which the metric is diagonal.
\end{Theorem}

Let $M$ be a three-dimensional spacetime and $x\in M$. Let $(U, \varphi)$ be a chart of $M$ around $x$ for which $g$ is diagonal. Let
\[
\begin{array}{cccc}
	\varphi^{-1}:&\varphi(U)&\to& U\\
	& (x_1,x_2,x_3)&\mapsto & \varphi^{-1}(x_1,x_2,x_3).
\end{array}
\]

The matrix representation of $g$ in the chart $(U, \varphi)$ is 
\begin{equation*}
	\label{matrixdiag}
	G\big(\varphi^{-1}(x_1,x_2,x_3)\big) = 
\begin{pmatrix}
	g_{11}(x_1,x_2,x_3) & 0 & 0\\
	0 & g_{22}(x_1,x_2,x_3)  & 0\\
	0 & 0 & g_{33}(x_1,x_2,x_3) 
\end{pmatrix}
\end{equation*}
for some smooth functions $g_{11}, g_{22}, g_{33}\in\mathcal{C}^\infty\big(\varphi(U)\big)$. Since the metric is non-degenerate at every point, two of such functions will be always positive and one will be always negative. Thus, we can assume $g_{11},g_{22}>0$ and $g_{33}<0$, without loss of generality. In addition, the fact that the metric is diagonal in this chart implies that the coordinate vector fields $u_i :=T\varphi^{-1}(e_i)$ give the eigendirections of the metric at every point. This implies that, for any $y\in U$, the fibre of the bundle of future null cones is
\[
C^+_y = \{\lambda\big(\frac{\cos{\theta}}{\sqrt{g_{11}(\varphi^{-1}(y))}}u_1(y)+\frac{\sin\theta}{\sqrt{g_{22}(\varphi^{-1}(y))}}u_2(y) + \frac{1}{\sqrt{-g_{33}(\varphi^{-1}(y))}}u_3(y)\big)\ |\ \lambda\in\mathbb{R}^+ \text{ and }\theta\in[0,2\pi)\},
\]
and the fibre over $y$ of the projectivised bundle can be identified with
\[
\mathcal{P}C_y\cong C_y^{+u}:=\{\frac{\cos{\theta}}{\sqrt{g_{11}}}u_1+\frac{\sin\theta}{\sqrt{g_{22}}}u_2 + \frac{1}{\sqrt{-g_{33}}}u_3\ |\ \theta\in[0,2\pi)\},
\]
where we have dropped the points in order to simplify the notation. This discussion allows us to define local coordinates on $\mathcal{P}C$ via
\[
\begin{array}{cccc}
	\Psi^{-1}: & \varphi(U)\times (0,2\pi) & \to & \Psi^{-1}\big(\varphi(U)\times (0,2\pi)\big)\\
	&(x_1,x_2,x_3,\theta)  & \mapsto & \frac{\cos\theta}{\sqrt{g_{11}}}u_1 + \frac{\sin\theta}{\sqrt{g_{22}}}u_2+\frac{1}{\sqrt{-g_{33}}}u_3\in T_{\varphi^{-1}(x_1,x_2,x_3)}M.
\end{array}
\]

Let us denote by $\partial_{x_1}, \partial_{x_2}, \partial_{x_3}, \partial_{\theta}$ the coordinate vector fields defined by $\Psi^{-1}$.

\begin{Proposition}
	\label{PropKernel}
	Following with the previous notation, the kernel $\mathcal{W}$ of the Engel distribution on $\mathcal{P}C$ defined by the Lorentz prolongation is spanned on $\Psi^{-1}\big(\varphi(U)\times (0,2\pi)\big)$ by the vector field
	\[
	Z = \frac{\cos\theta}{\sqrt{g_{11}}}\partial_{x_1} + \frac{\sin\theta}{\sqrt{g_{22}}}\partial_{x_2}+\frac{1}{\sqrt{-g_{33}}}\partial_{x_3}+(A\sqrt{g_{11}}\cos\theta+B\sin\theta\sqrt{g_{22}}-C\sqrt{-g_{33}})\partial_\theta,
	\]
	where we define
		\[
	\begin{cases}
		A &= \frac{1}{2g_{11}\sqrt{g_{11}g_{22}}}\frac{\partial g_{11}}{\partial x_{2}}+\frac{\sin\theta}{2 g_{11}\sqrt{-g_{11}g_{33}}}\frac{\partial g_{11}}{\partial x_3}
		\\
		B &= -\frac{1}{2g_{22}\sqrt{g_{11}g_{22}}}\frac{\partial g_{22}}{\partial x_{1}}-\frac{\cos\theta}{2 g_{22}\sqrt{-g_{22}g_{33}}}\frac{\partial g_{22}}{\partial x_3}\\
		C & =-\frac{\sin\theta}{2g_{33}\sqrt{-g_{11}g_{33}}}\frac{\partial g_{33}}{\partial x_1}+\frac{\cos\theta}{2g_{33}\sqrt{-g_{22}g_{33}}}\frac{\partial g_{33}}{\partial x_2}.
	\end{cases}
	\]
\end{Proposition}
\begin{proof}
	Recall that the  Engel structure $\mathcal{D}$ on $\mathcal{P}C$ defined by the Lorentz prolongation of $M$ is
	\[
	\mathcal{D}_{\Psi^{-1}(x_1,x_2,x_3,\theta)} = \mathcal{D}_{\frac{\cos\theta}{\sqrt{g_{11}}}u_1 + \frac{\sin\theta}{\sqrt{g_{22}}}u_2+\frac{1}{\sqrt{-g_{33}}}u_3} = (T\pi_L)^{-1}\Big(\lspan \frac{\cos\theta}{\sqrt{g_{11}}}u_1 + \frac{\sin\theta}{\sqrt{g_{22}}}u_2+\frac{1}{\sqrt{-g_{33}}}u_3\rspan\Big), 
	\]
	where $\pi_L:\mathcal{P}C\to M$ is the canonical projection. Let us compute
	\[
	T\pi_L(\partial_{x_1}) = \frac{d}{ds}\big|_{s = 0}\pi_L \circ \Psi^{-1}(x_1+s, x_2, x_3, \theta) = \frac{d}{ds}\big|_{s = 0}\varphi^{-1}(x_1+s, x_2, x_3) = u_1,
	\]
	and similarly for $T\pi_L(\partial_{x_2}) = u_2$ and $T\pi_L(\partial_{x_3}) = u_3$. In addition, 
	\[
	T\pi_L(\partial_\theta) = \frac{d}{ds}\big|_{s = 0}\pi_L \circ \Psi^{-1}(x_1, x_2, x_3, \theta +s) =  \frac{d}{ds}\big|_{s = 0}\varphi^{-1}(x_1, x_2, x_3)  = 0.
	\]
	
	Hence, by linearity, the rank-two distribution $\mathcal{D}$ is given pointwise by
	\[
	\mathcal{D} = \lspan X:=\frac{\cos\theta}{\sqrt{g_{11}}}\partial_{x_1} + \frac{\sin\theta}{\sqrt{g_{22}}}\partial_{x_2}+\frac{1}{\sqrt{-g_{33}}}\partial_{x_3}, \partial_\theta\rspan.
	\]
	
	If we define
	\[
	\dot{X} : = [\partial_\theta, X] = -\frac{\sin\theta}{\sqrt{g_{11}}}\partial_{x_1}+\frac{\cos\theta}{\sqrt{g_{22}}}\partial_{x_2},
	\]
	then the even-contact structure $\mathcal{E}$ on $\mathcal{P}C$ is 
	\[
	\mathcal{E} = \lspan X, \dot{X}, \partial_\theta\rspan.
	\]
	
	The next step is to compute the kernel $\mathcal{W}$ of the Engel structure. Let us first calculate
	\begin{align*}
		[X,\dot{X}] &= \Big[\frac{\cos\theta}{\sqrt{g_{11}}}\partial_{x_1} + \frac{\sin\theta}{\sqrt{g_{22}}}\partial_{x_2}+\frac{1}{\sqrt{-g_{33}}}\partial_{x_3},-\frac{\sin\theta}{\sqrt{g_{11}}}\partial_{x_1}+\frac{\cos\theta}{\sqrt{g_{22}}}\partial_{x_2} \Big]\\& = \cos^2\theta\Big[\frac{\partial_{x_1}}{\sqrt{g_{11}}}, \frac{\partial_{x_2}}{\sqrt{g_{22}}}\Big]-\sin^2\theta\Big[\frac{\partial_{x_2}}{\sqrt{g_{22}}}, \frac{\partial_{x_1}}{\sqrt{g_{11}}}\Big] -\sin\theta\Big[\frac{\partial_{x_3}}{\sqrt{-g_{33}}}, \frac{\partial_{x_1}}{\sqrt{g_{11}}}\Big]+\cos\theta\Big[\frac{\partial_{x_3}}{\sqrt{-g_{33}}}, \frac{\partial_{x_2}}{\sqrt{g_{22}}}\Big]\\& = \Big[\frac{\partial_{x_1}}{\sqrt{g_{11}}}, \frac{\partial_{x_2}}{\sqrt{g_{22}}}\Big]-\sin\theta\Big[\frac{\partial_{x_3}}{\sqrt{-g_{33}}}, \frac{\partial_{x_1}}{\sqrt{g_{11}}}\Big]+\cos\theta\Big[\frac{\partial_{x_3}}{\sqrt{-g_{33}}}, \frac{\partial_{x_2}}{\sqrt{g_{22}}}\Big]\\ &  = \frac{\partial }{\partial{x_1}}\Big(\frac{1}{\sqrt{g_{22}}}\Big)\frac{\partial_{x_2}}{\sqrt{g_{11}}} - \frac{\partial }{\partial{x_2}}\Big(\frac{1}{\sqrt{g_{11}}}\Big)\frac{\partial_{x_1}}{\sqrt{g_{22}}} -\sin\theta \Bigg(\frac{\partial}{\partial x_{3}}\Big(\frac{1}{\sqrt{g_{11}}}\Big)\frac{\partial_{x_1}}{\sqrt{-g_{33}}}- \frac{\partial}{\partial x_1}\Big(\frac{1}{\sqrt{-g_{33}}}\Big)\frac{\partial_{x_3}}{\sqrt{g_{11}}}\Bigg)\\&\phantom{=} + \cos\theta\Bigg(\frac{\partial}{\partial x_{3}}\Big(\frac{1}{\sqrt{g_{22}}}\Big)\frac{\partial_{x_2}}{\sqrt{-g_{33}}}- \frac{\partial}{\partial x_2}\Big(\frac{1}{\sqrt{-g_{33}}}\Big)\frac{\partial_{x_3}}{\sqrt{g_{22}}}\Bigg) \\ &= -\frac{1}{2g_{22}\sqrt{g_{11}g_{22}}}\frac{\partial g_{22}}{\partial x_{1}}\partial_{x_2} +\frac{1}{2g_{11}\sqrt{g_{11}g_{22}}}\frac{\partial g_{11}}{\partial x_{2}}\partial_{x_1} +\sin\theta\Bigg(\frac{1}{2 g_{11}\sqrt{-g_{11}g_{33}}}\frac{\partial g_{11}}{\partial x_3}\partial_{x_1}\\ &\phantom{=}-\frac{1}{2g_{33}\sqrt{-g_{11}g_{33}}}\frac{\partial g_{33}}{\partial x_1}\partial_{x_3}\Bigg)-\cos\theta\Bigg(\frac{1}{2 g_{22}\sqrt{-g_{22}g_{33}}}\frac{\partial g_{22}}{\partial x_3}\partial_{x_2}-\frac{1}{2g_{33}\sqrt{-g_{22}g_{33}}}\frac{\partial g_{33}}{\partial x_2}\partial_{x_3}\Bigg)\\ &= A\partial_{x_1}+B\partial_{x_2}+C\partial_{x_3},
	\end{align*}
	where we have defined
	\[
	\begin{cases}
	A &={} \frac{1}{2g_{11}\sqrt{g_{11}g_{22}}}\frac{\partial g_{11}}{\partial x_{2}}+\frac{\sin\theta}{2 g_{11}\sqrt{-g_{11}g_{33}}}\frac{\partial g_{11}}{\partial x_3}
	\\
	B &={} -\frac{1}{2g_{22}\sqrt{g_{11}g_{22}}}\frac{\partial g_{22}}{\partial x_{1}}-\frac{\cos\theta}{2 g_{22}\sqrt{-g_{22}g_{33}}}\frac{\partial g_{22}}{\partial x_3}\\
	C & ={}-\frac{\sin\theta}{2g_{33}\sqrt{-g_{11}g_{33}}}\frac{\partial g_{33}}{\partial x_1}+\frac{\cos\theta}{2g_{33}\sqrt{-g_{22}g_{33}}}\frac{\partial g_{33}}{\partial x_2}.
	\end{cases}
	\]
	
	Since we know that the kernel $\mathcal{W}$ lies within $\mathcal{D}$, there exist functions $\lambda, \mu\in\mathcal{C}^\infty(\mathcal{P}C)$ such that
	\[
	\mathcal{W} = \lspan \lambda X +\mu \partial_\theta\rspan.
	\]
	
	Let us compute
	\begin{align*}
		[\partial_{\theta}, \lambda X +\mu \partial_\theta] &= [\partial_\theta, \lambda X]+[\partial_\theta, \mu\partial_\theta] = \lambda\dot{X}+ \frac{\partial \lambda }{\partial\theta}X+\frac{\partial \mu}{\partial \theta}\partial_\theta \in\mathcal{E}
	\end{align*}
	regardless of $\lambda, \mu$. Similarly, 
	\[
	[X, \lambda X +\mu \partial_\theta] = [X, \lambda X]+[X, \mu\partial_\theta] = X(\lambda) X+X(\mu)\partial_\theta-\mu\dot{X} \in\mathcal{E}.
	\]
	
	Finally, 
	\begin{align*}
		[\dot{X}, \lambda X +\mu \partial_\theta] &= [\dot{X}, \lambda X]+[\dot{X}, \mu\partial_\theta] = \dot{X}(\lambda)X+\lambda[\dot{X}, X]+\dot{X}(\mu)\partial_\theta + \mu[\dot{X}, \partial_\theta] \\ & = \dot{X}(\lambda)X +\dot{X}(\mu)\partial_\theta-\lambda(A\partial_{x_1}+B\partial_{x_2}+C\partial_{x_3})-\mu\Big(-\frac{\cos\theta}{\sqrt{g_{11}}}\partial_{x_1}-\frac{\sin\theta}{\sqrt{g_{22}}}\partial_{x_2}\Big).
	\end{align*}
	
	Since $\dot{X}(\lambda)X+\dot{X}(\mu)\partial_\theta\in\mathcal{E}$, it is enough to impose that the last two terms belong to $\mathcal{E}$. Then, $-\lambda(A\partial_{x_1}+B\partial_{x_2}+C\partial_{x_3})-\mu\Big(-\frac{\cos\theta}{\sqrt{g_{11}}}\partial_{x_1}-\frac{\sin\theta}{\sqrt{g_{22}}}\partial_{x_2}\Big)\in \mathcal{E}$ if and only if
	\[
	-\lambda(A\partial_{x_1}+B\partial_{x_2}+C\partial_{x_3})-\mu\Big(-\frac{\cos\theta}{\sqrt{g_{11}}}\partial_{x_1}-\frac{\sin\theta}{\sqrt{g_{22}}}\partial_{x_2}\Big)+\lambda C\sqrt{-g_{33}}X\in \mathcal{E}. 
	\]
	
	The above vector field reads
	\[
	-\lambda\Bigg(\Big(A-C\cos\theta\sqrt{\frac{-g_{33}}{g_{11}}}\Big)\partial_{x_1}+\Big(B-C\sin\theta\sqrt{\frac{-g_{33}}{g_{22}}}\partial_{x_2}\Big)\Bigg)+\mu\Big(\frac{\cos\theta}{\sqrt{g_{11}}}\partial_{x_1}+\frac{\sin\theta}{\sqrt{g_{22}}}\partial_{x_2}\Big),
	\]
	and, except for isolated pathological cases, for it to belong to $\mathcal{E}$, it is enough to impose that it equals $\dot{X}$. This is equivalent to imposing 
	\[
	\begin{cases}
		\lambda \Big(A-C\cos\theta\sqrt{\frac{-g_{33}}{g_{11}}}\Big)-\mu\frac{\cos\theta}{\sqrt{g_{11}}} & ={}\frac{\sin\theta}{\sqrt{g_{11}}}\\
		\lambda\Big(B-C\sin\theta\sqrt{\frac{-g_{33}}{g_{22}}}\Big)-\mu\frac{\sin\theta}{\sqrt{g_{22}}} &={} -\frac{\cos{\theta}}{\sqrt{g_{22}}}.
	\end{cases}
	\]
	
	Whenever $\Big(A-C\cos\theta\sqrt{\frac{-g_{33}}{g_{11}}}\Big)\frac{\sin\theta}{\sqrt{g_{22}}}-\Big(B-C\sin\theta\sqrt{\frac{-g_{33}}{g_{22}}}\Big)\frac{\cos\theta}{\sqrt{g_{11}}}\neq 0$, the solution to the system is
	\[
\lambda = \frac{1}{A\sqrt{g_{11}}\sin\theta-B\sqrt{g_{22}}\cos\theta}\hspace{1cm}\text{ and }\hspace{1cm}\mu = \frac{A\sqrt{g_{11}}\cos\theta+B\sqrt{g_{22}}\sin\theta-C\sqrt{-g_{33}}}{A\sqrt{g_{11}}\sin\theta-B\sqrt{g_{22}}\cos\theta}.
	\]
	
	Hence, the kernel $\mathcal{W}$ is spanned by
	\[
	\mathcal{W} = \lspan \lambda X+\mu\partial_\theta\rspan = \lspan X+\frac{\mu}{\lambda}\partial_\theta\rspan = \lspan Z:=X+(A\sqrt{g_{11}}\cos\theta+B\sqrt{g_{22}}\sin\theta-C\sqrt{-g_{33}})\rspan.
	\]

	It can be seen that for the pathological cases for which $\Big(A-C\cos\theta\sqrt{\frac{-g_{33}}{g_{11}}}\Big)\frac{\sin\theta}{\sqrt{g_{22}}}-\Big(B-C\sin\theta\sqrt{\frac{-g_{33}}{g_{22}}}\Big)\frac{\cos\theta}{\sqrt{g_{11}}}$ vanishes, the kernel $\mathcal{W}$ is also spanned by $Z$. Hence, regardless of $A,B,C$ and $\theta$, we have
	\[
	\mathcal{W} = \lspan Z\rspan.
	\]
\end{proof}

\begin{Definition}
	\label{separable}
	A three-dimensional spacetime $(M, g)$ is said to be \textbf{separable} if, around every point $x\in M$, there exists a local chart $(U, \varphi)$ for which the matrix representation of $g$ is of the form 
	\[
		G\big(\varphi(x_1,x_2,x_3)\big) = 
	\begin{pmatrix}
		g_{11}(x_1,x_2,x_3) & 0 & 0\\
		0 & g_{22}(x_1,x_2,x_3)  & 0\\
		0 & 0 & g_{33}(x_1,x_2,x_3) 
	\end{pmatrix},
	\]
	where: 
	\begin{enumerate}
		\item $g_{11}, g_{22}>0$ and $g_{33}<0$,
		\item $g_{11}(x_1, x_2, x_3) = g_{11}(x_1, x_2)$ and $g_{22}(x_1, x_2, x_3) = g_{22}(x_1, x_2)$,
		\item $g_{33}(x_1,x_2,x_3) = g_{33}(x_3)$.
	\end{enumerate}
\end{Definition}
Note that we require the spatial components of the metric to be invariant under the flow of the negative eigendirection of $g$, and the time component of the metric to be invariant under the flow of any space-like vector field. Note that the spacetimes $(\S^2\times\S^1, g_c)$ studied in Section \ref{Model} are separable. For separable spacetimes, a straightforward computation using Proposition \ref{PropKernel} implies the following lemma.

\begin{Lemma}
	\label{lemmasep}
	Let $M$ be a separable three-dimensional spacetime. The kernel $\mathcal{W}$ on $\mathcal{P}C$ is spanned by
	\[
\mathcal{W} = \lspan Z = \frac{\cos\theta}{\sqrt{g_{11}}}\partial_{x_1}+\frac{\sin\theta}{\sqrt{g_{11}}}\partial_{x_2}+\frac{1}{\sqrt{-g_{33}}}\partial_{x_3} +\Big(\frac{\cos\theta}{2g_{11}\sqrt{g_{22}}}\frac{\partial g_{11}}{\partial{x_2}}-\frac{\sin\theta}{2g_{22}\sqrt{g_{11}}}\frac{\partial g_{22}}{\partial x_1}\Big)\partial_\theta\rspan.	
	\]
\end{Lemma}

We can now prove the most important result of the current section.

\begin{Theorem}
	\label{RecoverManifold}
	Let $(M,g)$ be a separable spacetime. Then
	\[
	\mathcal{N} = \mathcal{P}C/ \mathcal{W}.
	\]
\end{Theorem}
\begin{proof}
	Let $x\in M$ and let $(U,\varphi)$ be  a coordinate chart around $x$ satisfying the conditions of the definition of separable manifold. Let
	\[
	\begin{array}{cccc}
		\Psi^{-1}: & \varphi(U)\times (0,2\pi) & \to & \Psi^{-1}\big(\varphi(U)\times (0,2\pi)\big)\\
		&(x_1,x_2,x_3,\theta)  & \mapsto & \frac{\cos\theta}{\sqrt{g_{11}}}u_1 + \frac{\sin\theta}{\sqrt{g_{22}}}u_2+\frac{1}{\sqrt{-g_{33}}}u_3\in T_{\varphi^{-1}(x_1,x_2,x_3)}M.
	\end{array}
	\]
	be coordinates on $\mathcal{P}C$. By Lemma \ref{lemmasep}, we have 
	\[
\mathcal{W} = \lspan Z = \frac{\cos\theta}{\sqrt{g_{11}}}\partial_{x_1}+\frac{\sin\theta}{\sqrt{g_{11}}}\partial_{x_2}+\frac{1}{\sqrt{-g_{33}}}\partial_{x_3} +\Big(\frac{\cos\theta}{2g_{11}\sqrt{g_{22}}}\frac{\partial g_{11}}{\partial{x_2}}-\frac{\sin\theta}{2g_{22}\sqrt{g_{11}}}\frac{\partial g_{22}}{\partial x_1}\Big)\partial_\theta\rspan.	
\]
	
	Let us compute the integral lines of the vector field $Z$ on $\mathcal{P}C$. These integral lines are curves $\gamma(s) = \Psi^{-1}\Big(x_1(s),x_2(s), x_3(s), \theta(s) \Big)$ with
	\[
	\begin{cases}
		x_1'(s) &=\  \frac{\cos\theta}{\sqrt{g_{11}}}\\
		x_2'(s) &=\ \frac{\sin\theta}{\sqrt{g_{22}}} \\
		x_3'(s) & = \ \frac{1}{\sqrt{-g_{33}}}\\
		\theta'(s)  &=\ \frac{\cos\theta}{2g_{11}\sqrt{g_{22}}}\frac{\partial g_{11}}{\partial x_{2}}-\frac{\sin\theta}{2g_{22}\sqrt{g_{11}}}\frac{\partial g_{22}}{\partial x_{1}}
	\end{cases}
	\]
for $s\in(-\varepsilon, \varepsilon)$ for some $\varepsilon>0$.
	 Note that
	\begin{align*}
		x_1''(s) =& -\frac{\sin\theta}{\sqrt{g_{11}}}\Big(\frac{\cos\theta}{2g_{11}\sqrt{g_{22}}}\frac{\partial g_{11}}{\partial x_{2}}-\frac{\sin\theta}{2g_{22}\sqrt{g_{11}}}\frac{\partial g_{22}}{\partial x_{1}}\Big)+\cos\theta\Big(-\frac{1}{2g_{11}\sqrt{g_{11}}}\frac{\partial g_{11}}{\partial x_1}x'_1-\frac{1}{2g_{11}\sqrt{g_{11}}}\frac{\partial g_{11}}{\partial x_2}x'_2\Big)  \\ &= -\frac{1}{2 g_{11}}\frac{\partial g_{11}}{\partial x_2}x'_1x_2'+\frac{1}{2g_{11}}\frac{\partial g_{22}}{\partial x_1}(x'_2)^2-\frac{1}{2g_{11}}\frac{\partial g_{11}}{\partial x_1}(x'_1)^2-\frac{1}{2g_{11}}\frac{\partial g_{11}}{\partial x_2}x_1'x_2' \\ &=-\Gamma^1_{11}(x_1')^2 -2\Gamma_{12}^1x_1'x_2'-\Gamma_{22}^1(x_2')^2,
	\end{align*}
	which is precisely the geodesic equation in $M$ for $x_1$ under the assumptions that $\frac{\partial g_{11}}{\partial x_3} = \frac{\partial g_{22}}{\partial x_3} =\frac{\partial g_{33}}{\partial x_1} = \frac{\partial g_{33}}{\partial x_2}= 0$. It is clear that the same happens for $x_2$. Now, for the third coordinate,
	\begin{align*}
		x_3''(s) &= -\frac{1}{2 g_{33}\sqrt{-g_{33}}}\Big(\frac{\partial  g_{33}}{\partial x_3}x_3'\Big) = -(x_3')^2\big(-\frac{1}{2g_{33}}\frac{\partial g_{33}}{\partial x_3}\big) = -(x_3')^2\Gamma_{33}^3,
	\end{align*}
	and we also obtain the geodesic equation for $x_3$. Thus, if $\gamma:(-\varepsilon, \varepsilon)\to\Psi^{-1}\big(\varphi(U)\times (0,2\pi)\big)$ is an integral line of $Z$, the projection $\mu:=\pi_L\circ\gamma$ is a geodesic on $M$. Note that $\mu(s) = \varphi\big(x_1(s), x_2(s), x_3(s)\big)$ implies that
	\[
	\dot{\mu}(s) = x_1'u_1+x_2'u_2+x_3'u_3 = \frac{\cos\theta}{\sqrt{g_{11}}}u_1+\frac{\sin\theta}{\sqrt{g_{22}}}u_2+\frac{1}{\sqrt{-g_{33}}}u_3 = \gamma(s),
	\]
	and hence the integral lines of $Z$ are precisely the tangent vectors to the light geodesics of $M$. Since $x\in M$ is arbitrary and we can define an alternative parametrisation $\tilde{\Psi}^{-1}:\varphi(U)\times(-\pi, \pi)\to\tilde{\Psi}^{-1}(\varphi(U)\times(-\pi, \pi)),$ it is true globally that the integral lines of the kernel $\mathcal{W}$ are the tangent vectors to the light geodesics in $M$. Hence, the kernel is spanned by the restriction of the geodesic spray $X_g$ on $TM$ to the bundle of cones. Since clearly
	\[
	\mathcal{P}C = C^+/\Delta,
	\]
	where $\Delta$ is the Euler field, we have
	\[
	\mathcal{P}C/\mathcal{W} = \mathlarger{\mathlarger{ \sfrac{C^+/\Delta}{X_g}}} = \mathcal{N},
	\]
	as claimed.
\end{proof}

Hence, we can recover the space of light geodesics $\mathcal{N}$ of a separable spacetime $M$ from its Lorentz prolongation. The following results show that we can also recover the canonical contact structure on $\mathcal{N}$. Since the kernel $\mathcal{W}$ is always transverse to the coordinate $\partial_\theta$, the following proposition follows immediately.

\begin{Proposition}
	Let $x \in M$. Then, the sky of $x$ is
	\[
	\mathfrak{S}_x = \{p\big(\frac{\cos\theta}{\sqrt{g_{11}}}u_1+\frac{\sin\theta}{\sqrt{g_{22}}}u_2+\frac{1}{\sqrt{-g_{33}}}u_3 \big)\ |\ \theta\in[0,2\pi)\}, 
	\]
	where $p:\mathcal{P}C\to\mathcal{P}C/\mathcal{W}$ is the canonical projection.
	\label{Recoverctc}
\end{Proposition}
\begin{Theorem}
	\label{Bigthree}
	Let $M$ be a separable spacetime. Assume that the kernel $\mathcal{W}$ in $\mathcal{P}C$ is nice, that is, $\mathcal{P}C/ \mathcal{W}$ is a differentiable manifold and the canonical projection
	\[
	p:\mathcal{P}C\to\mathcal{P}C/\mathcal{W}
	\]
	is a submersion. Then, the canonical contact structure on $\mathcal{N} = \mathcal{P}C/\mathcal{W}$ is
	\[
	\mathcal{H} = p_*\mathcal{E}.
	\]
\end{Theorem}
\begin{proof}
	Firstly, as argued in \cite[p. 246]{adachi}, the even-contact structure $\mathcal{E}$ is invariant under the flow of any vector field generating $\mathcal{W}$. Therefore, the pushforward $p_*\mathcal{E}$ is well defined.

	Let $\gamma\in\mathcal{N} = \mathcal{P}C/ \mathcal{W}$. Then, $\gamma$ is defined by a curve $\mu: (-\varepsilon, \varepsilon)\to M$ which is a null geodesic in $M$, and 
	\[
	\gamma = p\big(\dot{\mu}(s)\big)
	\]
	for all $s\in(-\varepsilon, \varepsilon)$. Let $q_0 = \mu(0)$ and define coordinates $\varphi^{-1}:(x_1,x_2,x_3)\to\varphi^{-1}(x_1,x_2,x_3)$ around $q_0$ as defined in the proof of Theorem \ref{RecoverManifold}. Let also $\Psi^{-1}:(x_1,x_2,x_3, \theta)\to \Psi^{-1}(x_1,x_2,x_3, \theta)$ be coordinates around $\dot{\mu}(0)$ in $\mathcal{P}C$ as defined at the beginning of the current section. Now, for all $s\in[0,\varepsilon)$ small enough, the point $q_s: = \mu(s)$ lies in the image of $\varphi^{-1}$ and $\dot{\mu}(s)$ lies in the image of $\Psi^{-1}$.  Hence, we can define, for $s$ small enough, $\dot{\mu}(s) = \Psi^{-1}\big(x_1(s), x_2(s), x_3(s), \theta(s)\big)$. Also, if $s>0$ and $s$ is small enough, the points $q_0$ and $q_s$ are not conjugate in $M$.
	
	By the previous result, we have
	\[
	\mathfrak{S}_{q_0} = \{p\circ\Psi^{-1}\big(x_1(0), x_2(0), x_3(0), \theta\big)\ |\ \theta\in \big(\theta(0)-\pi, \theta(0)+\pi\big]\}
	\]
	and
	\[
	\mathfrak{S}_{q_s} = \{p\circ\Psi^{-1}\big(x_1(s), x_2(s), x_3(s), \theta\big)\ |\ \theta\in \big(\theta(s)-\pi, \theta(s)+\pi\big]\}.
	\]
	
	Hence, 
	\[
	T_\gamma \mathfrak{S}_{q_0} = T_{\dot{\mu}(0)}p(\lspan \partial_\theta\rspan)
	\]
	and
	\[
	T_\gamma \mathfrak{S}_{q_s} = T_{\dot{\mu}(s)}p(\lspan \partial_\theta\rspan ).
	\]
	
	Fix now $s>0$ small enough and let, for some open neighbourhood $V$ of $\dot{\mu}(s)$,
	\[
	\Phi_{-s}^Z: V \to  \Phi_{-s}^Z(V)
	\]
	be the flow at time $-s$ of the vector field 
	\[
	Z =  X + (A\sqrt{g_{11}}\cos\theta+B\sqrt{g_{22}}\sin\theta)\partial_\theta
	\]
	that generates the kernel $\mathcal{W}$. If we take $V$ small enough, then $\Phi^Z_{-s}$ is a diffeomorphism. In addition, $p\circ\Phi^Z_{-s} = p$. Hence, since $\Phi^Z_{-s}\big(\dot{\mu}(s)\big) = \dot{\mu}(0)$, we can compute
	\[
	T_\gamma\mathfrak{S}_{q_s} = T_{\dot{\mu}(s)}p(\partial_\theta) = T_{\dot{\mu}(0)}p\circ T_{\dot{\mu}(s)}\Phi^{Z}_{-s}(\partial_\theta)
	\]
	and
	\begin{align*}
		\mathcal{H}_{\gamma} &= T_\gamma \mathfrak{S}_{q_0} \oplus T_\gamma \mathfrak{S}_{q_s} = T_{\dot{\mu}(0)}p\big(\lspan \partial_\theta,  T_{\dot{\mu}(s)}\Phi^{Z}_{-s}(\partial_\theta) \rspan\big) = T_{\dot{\mu}(0)}p\big(\lspan \partial_\theta,  T_{\dot{\mu}(s)}\Phi^{Z}_{-s}(\partial_\theta)-\partial_\theta \rspan\big)  \\ &= T_{\dot{\mu}(0)}p\big(\lspan \partial_\theta, \frac{ T_{\dot{\mu}(s)}\Phi^{Z}_{-s}(\partial_\theta)-\partial_\theta}{s} \rspan\big)
	\end{align*}
	for all $s>0$ small enough. Hence, the result is still true if we take the limit $s\to 0$. Thus, we obtain
	\[
	\mathcal{H}_{\gamma}  = T_{\dot{\mu}(0)}p\big(\lspan \partial_\theta, \lim\limits_{s\to 0}\frac{ T_{\dot{\mu}(s)}\Phi^{Z}_{-s}(\partial_\theta)-\partial_\theta}{s} \rspan\big) =  T_{\dot{\mu}(0)}p\big(\lspan\partial_\theta,  [\partial_\theta, Z]\rspan\big).
	\]
	
	We can compute
	\begin{align*}
		[\partial_\theta, Z] = [\partial_\theta, X] + [\partial_\theta, (A\sqrt{g_{11}}\cos\theta+B\sqrt{g_{22}}\sin\theta)\partial_\theta] = \dot{X}+(-A\sqrt{g_{11}}\sin\theta+B\sqrt{g_{22}}\cos\theta )\partial_\theta.
	\end{align*}
	
	We obtain
	\begin{align*}
		\mathcal{H}_{\gamma}  &= T_{\dot{\mu}(0)}p\big(\lspan\partial_\theta,  [\partial_\theta, Z]\rspan\big) = T_{\dot{\mu}(0)}p\big(\lspan\partial_\theta,  \dot{X}+(-A\sqrt{g_{11}}\sin\theta+B\sqrt{g_{22}}\cos\theta )\partial_\theta\rspan\big) \\ & = T_{\dot{\mu}(0)}p\big(\lspan\partial_\theta,  \dot{X} \rspan\big) =  T_{\dot{\mu}(0)}p\big(\lspan\partial_\theta,  \dot{X}, Z \rspan\big)  = T_{\dot{\mu}(0)}p\big(\lspan\partial_\theta,  \dot{X}, X+ (A\sqrt{g_{11}}\cos\theta+B\sqrt{g_{22}}\sin\theta)\partial_\theta \rspan\big) \\ &=T_{\dot{\mu}(0)}p\big(\lspan\partial_\theta,  \dot{X}, X \rspan\big) = T_{\dot{\mu}(0)}p(\mathcal{E}).
	\end{align*}
	
	This concludes the proof that
	\[
	\mathcal{H} = p_*\mathcal{E}.
	\]
\end{proof}

\section{Recovering the Lorentzian Manifold }

After the discussion in Section \ref{Deprol}, it is natural to ask whether Engel structures can also be useful in recovering a Lorentzian manifold from its space of null geodesics. Two main problems arise when considering such approach. Firstly, it is not obvious how one can recover the direction $\partial_\theta$ of the Lorentz prolongation in terms of the Engel flag. Moreover, different spacetimes can define the same space of null geodesics and the same contact structure, as the next example shows.

\begin{Example}
	Let $M = \S^2\times \R$. Let $t$ be the coordinate in $\R$ and define the Lorentzian metric $g = g_\circ -dt^2$ on $M$. Following the same arguments as in Section \ref{LorentzianS2S1}, we can see that the space of null geodesics of $(M,g)$ is
	\[
	\mathcal{N} \cong ST\S^2.
	\]
	
	In addition, if $\gamma\in\mathcal{N}$, we can take $s_1,s_2\in\R$ such that $\pi_\R\big(\gamma(s_1)\big), \pi_\R\big(\gamma(s_2)\big)\in (-\pi, \pi)$, where $\pi_\R:M\to\R$ is the projection onto the second factor. Hence, the contact structure on $\mathcal{N}$ is exactly the same as the contact structure on $ST\S^2$ seen as the space of null geodesics of $\S^2\times \S^1$.
\end{Example}

A difference between the Lorentzian manifolds $\S^2\times\S^1$ and $\S^2\times\R$ with the proposed metrics is that, in the former, different points have different skies, while in the latter infinitely many points have the same sky. 

We present next a procedure that allows us to define a spacetime with a particular space of null geodesics and contact structure given a set of skies, which we define as follows.

\begin{Definition}
	\label{DefSetOfSkies}
	Let $(M, \xi)$ be a compact contact manifold of dimension 3. A \textbf{set of skies} $\Sigma$ on $(M, \xi)$ is a collection of subsets of $\Sigma$ such that
	\begin{enumerate}
		\item for all $S\in\Sigma$, there exists a diffeomorphism $\varphi:\S^1\to S$,
		\item for all $S\in \Sigma$, the subset $S$ is Legendrian, that is, everywhere tangent to the distribution $\xi$,
		\item for every $x\in M$ and $v\in \xi_x\setminus\{0\}$, there exists a unique $S\in \Sigma$ such that $x\in S$ and $T_xS = \lspan v\rspan$, as oriented lines.
	\end{enumerate}
\end{Definition}

Let us fix a compact contact three-manifold $(M,\xi)$ and a set of skies $\Sigma$ on $(M,\xi)$. Consider the Cartan prolongation $\pi_C:S(\xi)\to(M,\xi)$, with Engel flag $\mathcal{W}\subset\mathcal{D}\subset\mathcal{E}\subset TM$. Now, for every $S\in\Sigma$, let us parametrize it via the diffeomorphism
\[
\begin{array}{cccc}
	\varphi:&\S^1&\to&S\\
	&\theta & \mapsto & S(\theta),
\end{array}
\]
and consider the collection of embedded circles $\{\dot{S}(\S^1)\}_{S\in \Sigma}$ in $S(\xi)$. By point $iii)$ in Definition \ref{DefSetOfSkies}, these curves determine a foliation of $S(\xi)$. The leaves can be parametrized by
\[
\begin{array}{cccc}
		\dot{\varphi}:&\S^1&\to&\dot{S}\\
	&\theta & \mapsto & \dot{S}(\theta).
\end{array}
\]

 Let $\Theta$ be the rank-one distribution defined by the tangent spaces of the foliation. Let $L = S(\xi)/\Theta$ and $p:S(\xi)\to S(\xi)/\Theta$. 

Let us show that the space of leaves $L = S(\xi)/\Theta$ is Hausdorff. Indeed, take two different elements  $S_1,S_2\in S(\xi)/\Theta$, and let us make an abuse of notation by writing $S_1,S_2\in\Sigma$, meaning that $S_1 = p(\dot{S}_1) $ and $S_2 = p(\dot{S}_2)$. By definition, $\dot{S}_1$ and $\dot{S}_2$ are disjoint subsets of $S(\xi)$. In addition, since they are embedded circles in $S(\xi)$, they are compact. Since $S(\xi)$ is a smooth manifold, it is normal Hausdorff and hence there exist open subsets $U,V$ of $S(\xi)$ that separate $\dot{S}_1$ and $\dot{S}_2$. Hence, $S_1$ and $S_2$ are separable. Let us assume that $\Theta$ is nice, which implies that $L$ is a differentiable manifold and $p$ is a submersion. It is enough, for instance, for the foliation to be regular, see \cite[Prop. 11.4.2]{brickell}.

We claim that $\Theta\subset\mathcal{D}$. Indeed, let $v\in S(\xi)$ and take $S$ such that $v = \dot{S}(0)$, for a proper parametrisation of $S$. Then, the leaf of the foliation defined by $\Theta$ and containing $v$ is $\dot{S}$. Thus,
\[
(\pi_C)_*\big(\Theta_v\big) = (\pi_C)_*\big(\lspan \frac{d}{d\theta}\big|_{\theta = 0}\dot{S}(\theta)\rspan\big) = \lspan \frac{d}{d\theta}\big|_{\theta = 0} \pi_C\big(\dot{S}(\theta)\big)\rspan = \lspan \frac{d}{d\theta}\big|_{\theta = 0} S(\theta)\rspan = \lspan \dot{S}(0)\rspan = \lspan v\rspan,
\]
and the claim follows.

We will now define a metric on $L$. Recall that, when Lorentz prolonging a spacetime, the Engel distribution over a vector of the cone was spanned by the vector itself and the extra added coordinate. Hence, it is natural to define the null cone on a point $y :=p(\dot{S}(0))$ of $L$ as
\[
C_y = p_*\big(\mathcal{D}|_{\dot{S}}\big),
\]
that is, for every preimage $z\in p^{-1}(y)$, the pushforward $p_*\big(\mathcal{D}_z\big)$ defines a direction of the null cone over $y$. Since $\mathcal{D}$ is a rank-two distribution that contains the rank-one distribution $\Theta$, then $p_*\mathcal({D}_z)$ is indeed a one-dimensional vector subspace of $T_{p(z)}L$. The subset $C_y$ might not, in general, be a geometric cone, which we need to continue our discussion.  Let us also assume that the map $\theta\mapsto p_*\big(\mathcal{D}_{\dot{S}(\theta)}\big)$ is injective. The following lemma gives a characterisation of the Engel manifold $\mathcal{P}C$ that ensures this is the case. However, further research is needed to find more suitable hypotheses.

\begin{Lemma}
	\label{Lemmacones}
	Let $\dot{S}\subset S(\xi)$. Assume there exists an open subset $U$ of $S(\xi)$ containing $\dot{S}$ and vector fields $V,Y,Z\in\mathfrak{X}(U)$ such that $p_*V, p_*Y, p_*Z\in T_p(\dot{S}) L$ are constant over all $\dot{S}$, and for which
	\[
	\mathcal{D}_{\dot{S}(\theta)} = \lspan V\cos\theta + Y\sin\theta +Z\rspan\oplus\Theta. 
	\]
	
	Then, $C_{p(\dot{S})} = p_*\big(\mathcal{D}|_{\dot{S}}\big)$ is a cone and the map $\theta\mapsto p_*\big(\mathcal{D}_{\dot{S}(\theta)}\big)$  is injective.
\end{Lemma} 
\begin{proof}
	The cone on $y = p(\dot{S})$ is given by
	\[
	p_*\big(\mathcal{D}|_{\dot{S}}\big) = p_*\big(\{\lspan V\cos\theta + Y\sin\theta +Z\rspan\ |\ \theta\in \S^1\}\big) = \{p_*V\cos\theta+p_*Y\sin\theta+p_*Z\ |\ \theta\in\S^2\}, 
	\]
	which is clearly a geometric cone if $p_*V,p_*Y, p_*Z$ are constant over $\dot{S}$.
\end{proof}

\begin{Lemma}
	Let $M$ be a three-dimensional manifold and $C$ a bundle of cones over $M$ that vary smoothly with respect to the basepoint. Then, there exists a Lorentzian metric $g$ in $M$ whose bundle of cones is precisely $C$.
\end{Lemma}
\begin{proof}
	Let $x\in M$. Let $(U, \varphi)$ be a chart around $x$ and $\partial_{x_1}, \partial_{x_2}, \partial_{x_3}$ the coordinate vectors induced by the chart. For all $y\in U$, there exists a second-degree polynomial $p_y(z_1,z_2,z_3)$ such that $z_1\partial_{x_1}+z_2\partial_{x_2}+z_3\partial_{x_3}\in C_y$ if and only if $p_y(z_1,z_2,z_3) = 0$. Let us denote by $c_{ij} = c_{ji}$ the coefficient in $p_y$ of the monomial $z_iz_j$. Let now $G_y = (g_{ij})_y\in M_3(\R)$ be the symmetric matrix given by
	\begin{align*}[left = \empheqbiglbrace]
		g_{ii} &= c_{ii}\\
		g_{ij} & = \frac{c_{ij}}{2}, \text{ if $i\neq j$}.
	\end{align*}

By definition, a vector $z_1\partial_{x_1}+z_2\partial_{x_2}+z_3\partial_{x_3}\in C_y$ if and only if $(z_1,z_2,z_3)G_y(z_1,z_2,z_3)^t = 0$. Since $G_y$ represents a cone, necessarily the signature of $G_y$ is $(2,1)$. Since $G_y$ is diagonalisable and has three non-zero eigenvalues, $\det G_y\neq 0$. Thus, we can define $\tilde{G}_y = \frac{1}{|\det G_y|}G_y$. It is clear that $\tilde{G}_y$ still represents the cone and it is the only symmetric matrix with determinant $-1$ that does so. Since the cone $C_y$ varies smoothly with respect to $y$, so do the coefficients $c_{ij}$ and hence the matrix $\tilde{G}_y$. Thus, the collection $\{\tilde{G}_y\}_{y\in U}$ is a Lorentzian metric on $U$ with bundle of null cones $C|_U$.

Let $(V,\psi)$ be another local chart of $M$ with $U\cap V\neq\emptyset$ and $z\in U\cap V$. By uniqueness, the matrices $\tilde{G}_z$ and $\tilde{H}_z$ induced by the coordinate vectors of both charts are necessarily related by a change of basis matrix. Hence, the Lorentzian metrics defined in $U$ and $V$ coincide in $U\cap V$. Thus, taking an atlas of $M$, we can define a global Lorentzian matric with bundle of null cones $C$. 	
\end{proof}

Then, any other metric $g'$ with the same cones is conformal with $g$, and produces the same space of null geodesics, see \cite[Lemm. 2.1.2 and Prop. 2.1.3]{bautista}. Thus, we can fix any one such metric, say $g$. Since the foliation $ \{\dot{S}(\S^1)\}_{S\in\Sigma}$ is oriented, it induces an orientation on each of the cones $C_y$ on $L$ which globally gives an orientation of the bundle $C$. One can use the righ-hand rule to choose one of the two hemicones on each tangent space, and given that the orientation of $C$ is globally well-defined, this choice is globally consistent. Hence, the manifold $(L, g)$ is a spacetime. Assume it is separable. Further research is needed to find suitable hypotheses on the initial data that ensure that this is the case. Define now the Lorentz prolongation $\pi_L:\mathcal{P}C\to L$, with Engel flag $\tilde{\mathcal{W}}\subset\tilde{\mathcal{D}}\subset\tilde{\mathcal{E}}\subset T(\mathcal{P}C)$. The map
\[
\begin{array}{cccc}
	\Phi: & S(\xi) &\to & \mathcal{P}C\\
	& v & \to & p_*\big(\mathcal{D}_v\big)
\end{array}
\] 
provides a diffeomorphism. We claim that $\Phi_*(\mathcal{D})$ is precisely the canonical Engel structure $\tilde{\mathcal{D}}$ on $\mathcal{P}C$. Note that, if $v\in \S(\xi)$, then
\[
\pi_L\circ\Phi(v) = \pi_L\big(p_*(\mathcal{D}_v)\big) = p(v).
\]

Thus, $(\pi_L)_*\circ\Phi_*\big(\mathcal{D}_v\big) = (\pi_L\circ \Phi)_*\big(\mathcal{D}_v\big)= p_*\big(\mathcal{D}_v\big) = \Phi(v)$, and hence $\Phi_*\mathcal{D}\subset\tilde{\mathcal{D}}$. It is only left to show that $\Phi_*(\mathcal{D})$ is a rank-two distribution. Since $\Phi$ is a diffeomorphism, it is in particular a submersion and the claim follows.

Thus, if $\mathcal{N}$ is the space of null geodesics of $(L,g)$, we have, by Theorem \ref{RecoverManifold},
\[
\mathcal{N} = \mathcal{P}C/\tilde{\mathcal{W}}\cong S(\xi)/\mathcal{W}\cong M,
\]
and we recover the initial manifold $M$. In addition, if all the projections involved are submersions, the canonical contact structure $\mathcal{H}$ on $\mathcal{N}$ is, by Theorem \ref{Recoverctc},
\[
q_*\tilde{\mathcal{E}},
\]
where $q:\mathcal{P}C\to \mathcal{P}C/\tilde{\mathcal{W}}$ is the canonical projection. This contact structure gets carried to $M$ as
\[
(\pi_C)_*\mathcal{E}= \xi
\]
and hence we also recover the contact structure on $M$.

\chapter*{\huge List of Notations}
\addcontentsline{toc}{chapter}{List of Notations}
\vspace{-2cm}
\centering
\begin{tabular}{cp{0.8\textwidth}}
	$M$ & Differentiable manifold\\
	$TM$ & Tangent bundle of a manifold\\
	$\xi$ & Field of hyperplanes \\
	$\mathfrak{X}(M)$ & Space of smooth vector fields on a manifold\\
	$X, Y, Z, V$ & Vector field\\
	$\lspan -,-\rspan$ & Span of vectors, span of vector fields\\
	$\alpha, \theta$ & Differential form on a manifold, curve in a manifold \\
	$\ker$ & Kernel of a differential form, kernel of a linear map\\
	$TM/\xi$ & Quotient bundle\\
	$d\alpha$  & Differential of a form\\
$\wedge$ & Exterior product of forms\\
	$[-,-]$ & Lie bracket, distribution generated by the Lie brackets of other distributions\\
	$\iota$ & Inner product of a differentiable form and a vector field, inclusion map\\
	$\R^n$ & Euclidean space of dimension $n$\\
	$STM$ & Unit tangent bundle of a manifold\\
	$^\perp$ & Orthogonal subspace in a tangent space\\
	$T^*M$ & Cotangent bundle of a manifold\\
	$\S^n$ & Sphere of dimension $n$ \\
	$\tilde{\pi}$ & Projection of the cotangent bundle onto the manifold\\
	$\pi$ & Ratio of circle's perimeter to its diameter, projection from the tangent bundle to the manifold\\
	$Tf$ & Tangent map of a smooth map\\
	$\partial_x$ & Coordinate vector field\\
	$\delta_{ij}$ & Kronecker delta\\
	$dx$ & Coordinate one-form\\
	$f_*$ & Pushforward of a smooth map\\
	$\chi$ & Canonical contact structure on a tangent manifold \\
	$g$ & Pseudo-Riemannian metric\\
\end{tabular}

\centering
\begin{tabular}{cp{0.8\textwidth}}
		$G, (g_{ij})$ & matrix representation of a pseudo-Riemannian metric\\
		$(-,-)$ & Signature of a pseudo-Riemannian metric\\
		$\nabla$ & Affine connection on a manifold, Levi-Civita connection on a pseudo-Riemannian manifold\\
		$\Gamma_{ij}^k$ & Christoffel symbols of a pseudo-Riemannian manifold in local chart\\
	$\gamma, \mu,\beta, c$ & Curve within a manifold, geodesic in a pseudo-Riemannian manifold\\
	$\frac{D}{dt}$ & Covariant derivative in a manifold\\
	$\dot{\alpha},\dot{\theta}, \dot{\gamma}, \dot{\mu}, \dot{\beta}, \dot{c}$ & Tangent vector on a curve\\
	$\mathcal{I}C$ & Bundle of non-space-like vectors of a Lorentzian manifold\\
	$C$ & Bundle of null vectors of a Lorentzian manifold\\
	$C^+$ & Bundle of future pointing null vectors of a Lorentzian manifold\\
	$\mathcal{N}$ & Space of null geodesics of a spacetime\\
	$X_g$ & Geodesic spray on the tangent bundle of a Lorentzian manifold\\
	$\Delta$ & Euler field on the tangent bundle of a Lorentzian manifold\\
	$\mathcal{D}$ & Rank-two distribution, Engel structure\\
	$M/\xi$ & Space of leaves of a distribution, orbit space \\
	$\pi_\mathcal{N}$ & Projection of $C^+$ onto $\mathcal{N}$\\
	$\mathfrak{S}$ & Sky of a point\\
	$\mathcal{H}$ & Canonical contact structure on the space of null geodesics\\
	$\mathcal{E}$ & Even-contact structure\\
	$\mathcal{W}$ & Kernel of an Engel structure\\
	$\theta$ & Differential form on a manifold\\
	$B^3$ & Ball of dimension 3 in $\R^3$\\
	$\pi_C:S(\xi)\to M$ & Cartan prolongation of a contact manifold\\
	$p$ & Canonical projection onto the space of leaves\\
	$\mathcal{P}C$ & Bundle of projectivised cones of a Lorentzian manifold\\
	$\pi_L:\mathcal{P}C\to M$ & Lorentz prolongation of a Lorentzian manifold\\
	$\langle -,-\rangle$ & Inner product on Euclidean space, inner product in the division algebra of quaternions \\
	$g_\circ$ & Round metric on the sphere \\
	$\iota'$ & Inclusion map\\
	$\pi_M:M\times N\to M$ & Projection from a product manifold onto a factor\\
	$\Z_c$ & Finite cyclic group of order $c$\\
	$\H$ & Division algebra of quaternions\\
	$\C$ & Field of complex numbers\\
	$i$ & Complex unit in $\C$, complex unit in $\H$\\
	$j,k$ & Complex unit in $\H$\\
	$\V$ & Vector space of pure imaginary quaternions\\
			$S\H$ & Space of unit quaternions\\
\end{tabular}

\centering
\begin{tabular}{cp{0.8\textwidth}}
						$S\V$ & Space of unit pure imaginary quaternions\\
	$f$ & Projection of $ST\S^2$ onto $\S^2$ via the cross product\\
	$\tau_\omega:S\H\to S\V$ & Hopf fibration induced by unit pure imaginary quaternion $\omega$\\
	$L(p,1)$ & Lens space \\
	$\R P^n$ & Projective space of dimension $n$\\
	$r$ & Projection of $ST\S^2$ onto a lens space\\
	$\sigma$ & Antipodal map on $\S^n$\\
	
\end{tabular}

\end{document}